\documentclass[11pt,a4paper]{amsart}		

\usepackage[applemac]{inputenc}					% for unicode support

\usepackage{lmodern}						% font
\usepackage[english]{babel}					% babel
\usepackage[weather]{ifsym}	
\usepackage{mathrsfs}				% for visible angle

\usepackage{hyperref,enumitem,soul}

\hypersetup{
	colorlinks,
	linkcolor={red!50!black},
	citecolor={blue!50!black},
	urlcolor={blue!80!black}
}

\usepackage{tikz}
\usetikzlibrary{calc, svg.path, patterns}
\usepackage{todonotes}

\usepackage{amsmath,amsfonts,amssymb,amsthm}
\usepackage[left=3cm,right=3cm,top=3cm,bottom=3cm]{geometry}

\usepackage{mathtools}						% uncomment to tag only the referred equations (post-production)
\theoremstyle{plain}
\newtheorem{theorem}{Theorem}
\newtheorem{proposition}{Proposition}[section]

\newtheorem{lemma}[proposition]{Lemma}

\newtheorem{remark}[proposition]{Remark}

\numberwithin{equation}{section}

\newcommand{\R}{\mathbb{R}}					% real numbers
\newcommand{\C}{\mathbb{C}}					% complex numbers
\newcommand{\N}{\mathbb{N}}					% positive integers

\def\norm#1#2{\|#1\|_{#2}}

\def\refer#1{~\ref{#1}}
\def\refeq#1{~(\ref{#1})}
\def\ccite#1{~\cite{#1}}

\def\suite#1#2#3{(#1_{#2})_{#2\in {#3}}}

\def\inte#1{
\displaystyle\mathop{#1\kern0pt}^\circ }

\def\sumetage#1#2{\sum_{\substack{{#1}\\{#2}}}}

\def\supetage#1#2{\sup_{\substack{{#1}\\{#2}}}}

\let\d=\delta

\let\lam=\lambda

\let\D=\Delta

\let\S=\Sigma

\let\wt=\widetilde
\let\wh=\widehat

\def\cA{{\mathcal A}}

\def\cD{{\mathcal D}}

\def\cF{{\mathcal F}}

\def\cS{{\mathcal S}}

\def\cW{{\mathcal W}}
\def\cX{{\mathcal X}}

\def\S{{\mathop{\mathbb  S\kern 0pt}\nolimits}}

\def\virgp{\raise 2pt\hbox{,}}
\def\cdotpv{\raise 2pt\hbox{;}}
\def\eqdefa{\buildrel\hbox{\footnotesize def}\over =}

\def\C{\mathop{\mathbb C\kern 0pt}\nolimits}
\def\EE{\mathop{{\mathbb E \kern 0pt}}\nolimits}
\def\K{\mathop{\mathbb K\kern 0pt}\nolimits}
\def\N{\mathop{\mathbb N\kern 0pt}\nolimits}
\def\Q{\mathop{\mathbb Q\kern 0pt}\nolimits}
\def\R{{\mathop{\mathbb R\kern 0pt}\nolimits}}
\def\SS{\mathop{\mathbb S\kern 0pt}\nolimits}
\def\ZZ{\mathop{\mathbb Z\kern 0pt}\nolimits}
\def\TT{\mathop{\mathbb T\kern 0pt}\nolimits}
\def\P{\mathop{\mathbb P\kern 0pt}\nolimits}
\def \H{{\mathop {\mathbb H\kern 0pt}\nolimits}}
\newcommand{\ds}{\displaystyle}

\newcommand{\Z}{{\ZZ}}

\newcommand{\beq}{\begin{equation}}
\newcommand{\eeq}{\end{equation}}
\newcommand{\ben}{\begin{eqnarray}}
\newcommand{\een}{\end{eqnarray}}
\newcommand{\beno}{\begin{eqnarray*}}
\newcommand{\eeno}{\end{eqnarray*}}
\newcommand{\bqs}{\begin{equation*}}
\newcommand{\eqs}{\end{equation*}}
\newcommand{\andf}{\quad\hbox{and}\quad}
\newcommand{\with}{\quad\hbox{with}\quad}

\def \cFH {\cF_\H}

\def\equivH#1 {\buildrel\hbox{\tiny {$#1$}}\over \equiv}
\def\simH#1 {\buildrel\hbox{\footnotesize {$#1$}}\over \sim}

\title[Local dispersive and Strichartz estimates for the Schr\"odinger operator]
{Local dispersive  and Strichartz  estimates for the Schr\"odinger operator on the Heisenberg group}

\date{\today}

\author[H. Bahouri]{Hajer Bahouri}
\address[H. Bahouri]
{CNRS  \&  Sorbonne Universit\'e  \\
 Laboratoire Jacques-Louis Lions (LJLL) UMR  7598 \\
4, Place Jussieu\\
75005 Paris, France.}
\email{hajer.bahouri@ljll.math.upmc.fr}
\author[I. Gallagher]{Isabelle Gallagher}
\address[I. Gallagher]%
{DMA, \'Ecole normale sup\'erieure, CNRS, PSL Research University, 75005 Paris 
 \\
and UFR de math\'ematiques, Universit\'e de Paris, 75013 Paris, France.}
\email{gallagher@math.ens.fr}
\begin{document}
\setstcolor{red}

\begin{abstract} 
It was proved by H. Bahouri, P. G\'erard and  C.-J. Xu in\ccite{bgx} that the Schr\"odinger equation on  the Heisenberg group~$\H^d$,  involving the sublaplacian,  is an example of a  totally non-dispersive evolution equation: for this reason  global dispersive   estimates cannot hold. This paper aims at establishing  local dispersive   estimates on~$\H^d$  for the linear Schr\"odinger   equation, by a refined study of  the Schr\"odinger kernel~$ S_t$  on~$\H^d$.  The sharpness of these estimates  is discussed through several examples.
Our approach,     based on  the explicit formula  of   the heat kernel  on~$\H^d$ derived  by  B. Gaveau in\ccite{B. Gaveau},  is achieved by   combining  complex analysis and   Fourier-Heisenberg  tools. As a by-product of our results, we  establish local Strichartz  estimates    and prove  that the kernel~$ S_t$ concentrates  on quantized horizontal hyperplanes of~$\H^d$. 
\end{abstract}

\maketitle

\setcounter{tocdepth}{1}
 \tableofcontents

\noindent {\sl Keywords:}  Heisenberg group, Schr\"odinger equation, dispersive   estimates, 
Strichartz   estimates.

\vskip 0.2cm

\noindent {\sl AMS Subject Classification (2000):} 43A30, 43A80.
\section{Introduction}\label {intro}
\setcounter{equation}{0}
\subsection{Setting of the problem}\label {introst} It is well-known  that the solution to the free Schr\"odinger equation on~$\R^{n}$ 
$$
(S)\qquad \left\{
\begin{array}{c}
i\partial_t u -\D u = 0\\
u_{|t=0} = u_0
\end{array}
\right. $$
can be explicitly written with a convolution kernel  for $t \neq 0$   
\beq
\label {kernelSR}  u(t,\cdot)=  u_0 \star   \frac {{\rm e}^{-i \frac {|\cdot|^2} {4t} }} {(-4 \pi i t)^\frac {n} 2} \, \cdotp \eeq
 The proof of this explicit representation stems by a combination of Fourier and complex analysis arguments, from the expression   of the heat kernel on~$\R^{n}$. More precisely, taking the partial Fourier transform of~$(S)$ with respect to the variable $x$  
and   integrating in time the resulting ODE, we get
$$
 {\wh u} (t,\xi) =  {\rm e}^{i t |\xi |^2 }{\wh u_0} (\xi)\, ,
 $$ 
 where for any function~$g \in L^1(\R^n)$ we have defined
$$
\wh g(\xi) \eqdefa  \cF (g)(\xi) \eqdefa   \int_{\R^n} {\rm e}^{-i\langle x,\xi\rangle}g(x) \, dx \, .
$$
 The heart of the matter to prove~(\ref{kernelSR}) then consists  in computing   in the sense of distributions  the inverse Fourier transform of the complex Gaussian 
\beq
\label {kernelformula}  (\cF^{-1}{\rm e}^{i t |\cdot |^2 }) (x)= \frac {{\rm
e}^{-i \frac {|x|^2} {4t} }} {(-4 \pi i t)^\frac {n} 2} \, \cdotp \eeq
The proof of Formula\refeq{kernelformula}  is    based on two observations: first,  that for any~$x$ in~$\R^n$, the  two  maps 
\[
z \in\C \longmapsto H_1(z)\eqdefa \frac1{(2\pi)^n} \int_{\R^n} {\rm e}^{i\langle x,\xi\rangle}{\rm e}^{-z|\xi|^2}\,d\xi \quad\hbox{and}\quad
z \in\C\longmapsto H_2(z)\eqdefa \frac  1 {\big(4 \pi z \big)^{\frac n 2}} {\rm e}^{-\frac {|x|^2} {4z}}
\]
are  holomorphic  on the domain~$D$ of complex numbers with positive real part.  Accordingly   with  the expression of the heat kernel, 
 these two functions coincide on the intersection of the real line with~$D$, and thus they    coincide on the whole domain~$D$. Second, if~$\suite z p \N$ denotes a
sequence of elements of~$D$ which converges to~$-it$ for $t\not =0$,  the use of   the Lebesgue dominated convergence  theorem ensures that $H_1(z_ p )$ and $H_2(z_ p )$ converge in~$\cS'(\R^n)$,  as~$p$ tends to infinity, which   achieves the proof of~(\ref{kernelformula}).

 \medskip
 Formula~(\ref{kernelSR})   implies by Young's inequality the following dispersive estimate:
\beq
\label {dispSR}
\forall t \neq 0 \, , \quad \|u(t,\cdot)\|_{L^\infty(\R^{n})} \leq \frac 1 {(4\pi|t|)^{\frac n 2}} \|u_0\|_{L^1(\R^{n})}\,\cdot
\eeq
Such estimate plays a key role  in the study  of  semilinear and quasilinear equations  which appear in numerous physical applications.  Combined with an abstract functional analysis argument known as the $TT^*$-argument, it  yields a range of inequalities involving space-time Lebesgue norms, known as  Strichartz estimates\footnote{For further details, one can consult the papers of Ginibre-Velo~\cite{ginibrevelo},   Keel-Tao\ccite{keeltao} and Strichartz\ccite{strichartz}.}. When $u_0$ is for instance in~$L^2(\R^{n})$,  the above dispersive estimate\refeq{dispSR}  gives rise to the following Strichartz estimate  for the solution to the free Schr\"odinger equation
\beq
\label {dispSTR}
\|u\|_{L^q(\R; L^p(\R^{n}))} \leq C(p,q)   \|u_0\|_{L^2(\R^{n})}\,,
\eeq
where $(p,q)$  satisfies the scaling admissibility condition
\beq
\label {admissibR}
\frac 2 q + \frac n p = \frac n 2 \with q \geq 2 \andf   (n,q,p) \neq (2,2, \infty) \, .
\eeq
The interest for this issue has soared in the last decades.   We refer for instance to the monographs\ccite{BCD1, tao} and the references therein for an overview on this topic in the euclidean framework.   

\medskip In the present work, we aim at investigating  this phenomenon for the Schr\"odinger equation  on the Heisenberg group~$\H^d$   involving the sublaplacian.  Recall that in\ccite{bgx}, the first author along with P. G\'erard and C.-J. Xu  proved that no dispersion  occurs for this equation, and in particular  exhibited an example for which the Schr\"odinger operator   on~$\H^d$  behaves as  a transport equation with respect to one direction, known as the vertical direction.    
More precisely   they established the  following result   which shows that
a global dispersive estimate  of the type\refeq{dispSR} cannot be expected on~$\H^d$. We refer to the coming paragraph for the notation. 
\begin{proposition} [\cite{bgx}]  
\label {nodispersionS}
{\sl
There exists a function~$u_0$ in the Schwartz class~$\cS(\H^{d})$ such that the solution to  the free Schr\"odinger equation on~$\H^d$   satisfies
\beq \label {cex1}
\forall t \in \R \, , \quad
 \forall (Y,s) \in \H^d \, , \quad u(t,Y,s) = u_0(Y,s+4td)\, .
\eeq
}
\end{proposition}
This   result rules out an estimate of the type\refeq{dispSR}   in the setting of the Heisenberg group but  does not exclude a degraded estimate:  for instance in the case when~$u_0$ is compactly supported, then the solution remains compactly supported, in a set transported along the vertical line, so a local $L^\infty$ norm decays to zero with time.
Inspired by  the euclidean strategy displayed above, we shall indeed be able to establish local   decay in the spirit of\refeq{dispSR}.  The precise result is stated in the next paragraph.  As in the euclidean case, such a local dispersive estimate stems from the  explicit expression of  the Schr\"odinger kernel~$ S_t$   on~$\H^d$,  which turns out to be of type\refeq{kernelformula} in a  horizontal strip of~$\H^d$ (see Theorem\refer{STth}).

\medskip  Note also that a lack of dispersion was highlighted for the Schr\"odinger  propagator (associated  with the sublaplacian) in the framework of  H-type groups (\cite{hiero})  or more generally  in the case of  $2$-step    stratified Lie groups (\cite{bfg}).  
  More precisely, if~$p$ denotes   the dimension of the center of the H-type group, M. Del Hierro proved  in \cite{hiero}  sharp dispersive inequalities   for the Schr\"odinger  equation solution (with a~$ | t   |^{-(p-1)/2}$ decay). 
Concerning the more general    case of~$2$-step    stratified Lie groups,  the authors along with C. Fermanian-Kammerer~\cite{bfg}   emphasized   the  key  role played by the canonical skew-symmetric form   in determining the rate of decay of the solutions  of  the Schr\"odinger  equation: they established that if~$p$ denotes the dimension of  the center of  a~$2$-step    stratified Lie group $G$ and~$k$ the dimension of the   radical of its canonical skew-symmetric form, then the solutions  of  the Schr\"odinger  equation on $G$  satisfy  dispersive estimates with a   rate of decay at most  of order~$ | t   |^{- \frac{k+p-1} 2}$.

\subsection{Basic facts about the  Heisenberg group}\label {introbasic} Recall  that the~$d$-dimensional Heisenberg group~$\H^d$ can be defined as~$T^\star\R^d \times\R$   where~$T^\star\R^d$ is  the cotangent bundle,    endowed  with the noncommutative  product law\footnote{We refer to the monographs\ccite{bfgpseudo, fisher, folland, stein2, taylor1, 
thangavelu} and the references therein for further details.}
\begin{equation}\label{lawbis} 
(Y,s)\cdot (Y',s') \eqdefa \bigl(Y+Y',   s+s'+2  \langle\eta,y'\rangle-2\langle\eta',y\rangle\bigr)\, ,
\end{equation}
where
  $w=(Y,s)=(y,\eta,s)$ and $w'=(Y',s')=(y',\eta',s')$ are    elements of~$\H^d$. The variable~$Y$ is called the horizontal variable, while the variable $s$ is known as the vertical variable.

 \medskip  
 The space $\H^d$ is provided with a smooth left invariant measure, the Haar measure, which in  
the coordinate system $(Y, s)$ is simply the Lebesgue measure. 
  In particular, one can define   the following (noncommutative) convolution product  for any two integrable functions~$f$ and~$g$:
\beq
\label {defconvH}f \star g ( w ) \eqdefa \int_{\H^d} f ( w \cdot v^{-1} ) g( v)\, dv 
= \int_{\H^d} f ( v ) g( v^{-1} \cdot w)\, dv\, ,
\eeq
and   the usual Young inequalities are  valid:
\beq
\label {definyoungH}\norm{ f \star g }{L^r(\H^d)} \leq \norm f {L^p(\H^d)} \norm g {L^q(\H^d)} ,
 \,  \,   \hbox{whenever}\, 1\leq p,q,r\leq\infty\, \hbox{ and }\,
 \frac{1}{r} =  \frac{1}{p} + \frac{1}{q} - 1\, .
\eeq
The dilation on~$\H^d$ is defined
for  all $a > 0$ by
\beq \label{defdilation}
\delta_a ( Y,  s ) \eqdefa ( a Y,  a^2s )\, .
\eeq
Since, for all $a > 0$ and any~$f \in L^1(\H^d)$,    
$$
\int_{\H^d} f\big (\delta_a (w)\big)  dw= a^{-(2d+2)}\int_{\H^d} f (w)  dw \, ,
$$  
the  homogeneous dimension of~$\H^d$  is~$Q\eqdefa 2d+2$. 

\medskip

 The natural distance  on $\H^d$ compatible with the product law\refeq{lawbis}  is  called  the Kor\'anyi  distance and is defined by 
\beq
\label {distH}d_\H (w,w') 
\eqdefa \rho_\H(w^{-1}\cdot w')\, , 
\eeq
for all $w$, $w'$ in $ \H^d$, where $\rho_\H$ stands for the distance to the origin
\beq
\label {origdist}\rho_\H(w)= \rho_\H(Y,s) \eqdefa\bigl(|Y|^4 +  s^2\bigr)^{\frac 1 4}
\, . \eeq
In the following~$B_\H(w_0, R)$  denotes    the Heisenberg ball centered at $w_0$ and of radius~$R$ for the distance~$d_\H$ defined by\refeq{distH}, namely 
$$ B_\H(w_0, R) \eqdefa \Big\{w\in\H^d  \, / \, d_\H (w,w_0) < R  \Big\} \,  .$$ 
Observing that the distance $d_\H$ is invariant by left translation, that is to say
$$
\forall (w,w',w_0) \in (\H^d)^3, \  d_\H\left(\tau_{w_0}
(w),  \tau_{w_0}(w' )\right)= d_\H(w,w') 
$$
where~$\tau_{w_0}$ denotes  the  left translation  defined by  
\beq\label {definlefttranslate}
\tau_{w_0}(w)\eqdefa w_0\cdot w\, ,\eeq
one can readily check that 
$ \tau_{w_0}\big (B_\H(0, R)\big) = B_\H(w_0, R)$. 

 \medskip
Most classical analysis tools of~$\R^n$ can be adapted to~$\H^d$, resorting to the following left invariant vector fields 
$$
 \cX_j\eqdefa\partial_{y_j} +2\eta_j\partial_s\andf \Xi_j\eqdefa \partial_{\eta_j} -2y_j\partial_s \with j\in\{1,\dots,d\}\, ,
$$
known as the horizontal left invariant vector fields. In particular,  the sublaplacian is given    by 
$$
 \D_{\H}  \eqdefa \sum_{j=1} ^d (\cX_j^2+\Xi_j^2)\, .
$$
   For instance the Schwartz space~$\cS(\H^d)$, which is nothing else than~$\cS(\R^{2d+1})$, can be characterized by means  of~$\D_{\H}$ and~$\rho_\H$.

     \subsection{Main results}\label {disp}

  The main goal  of this article is to establish local dispersive estimates  for the free linear Schr\"odinger  equation  on $\H ^d$  associated  with the sublaplacian
$$
(S_\H)\qquad \left\{
\begin{array}{c}
i\partial_t u -\D_\H u = 0\\
u_{|t=0} = u_0 \,.
\end{array}
\right. 
$$  
 As in the euclidean case, one can readily establish that the Cauchy problem $(S_\H)$ admits a unique,  global in time solution if~$u_0 \in L^2(\H^d)$, by resorting to    Fourier-Heisenberg  analysis tools   or to functional calculus of the self-adjoint operator $-\D_\H$ (see Section\refer{KS} for further details). 
 Denoting by~$(\mathcal U(t))_{t\in \R} $ the solution operator, namely~$\mathcal U(t)u_0$ is the solution of~$(S_\H)$ at time $t$ associated with the data~$u_0$, then similarly to the euclidean case   $(\mathcal U(t))_{t\in \R} $  is a one-parameter group of unitary operators on $L^2(\H^d)$.

\medskip
The first result we establish  states as follows.
\begin{theorem}
\label{mainth1}
{\sl Given~$w_0 \in \H^d$, let $u_0$ be a function   in $\cD(B_\H(w_0, R_0))$.
Then the solution to the Cauchy problem~$(S_\H)$ associated to $u_0$ 
disperses locally  for large $|t |$, in the sense that, for any positive  constant~$\kappa< \sqrt{4d}$,     the following estimate holds for all $2 \leq p \leq \infty$:
\beq
\label {dispSHloc}
\|u(t,\cdot)\|_{L^p(B_\H(w_0, \kappa {|t |^\frac12}))} \leq
\left( \frac {M_\kappa}  {|t |^{\frac Q 2} }\right)^ {1- \frac 2 p} \|u_0\|_{L^{p'}(\H^{d})}\, \virgp
\eeq for all $|t | \geq T_{\kappa, R_0}$,   where  
\begin{equation} \label{time}
\ds T_{\kappa, R_0}\eqdefa \Big(\frac {R_0} {\sqrt{4d}-\kappa} \Big)^2  \andf  M_\kappa \eqdefa \frac {1}  {(4 \pi)^{\frac Q 2} } \int_{\R} \Big(\frac {2  \tau}  {\sinh 2 \tau}\Big)^{d} \exp\Big(\frac { \kappa^2\tau} 2\Big)d\tau\,,
\end{equation} 
and~$p'$ is the conjugate exponent to~$p$. }
\end{theorem}
  \begin{remark}
\label{mainth1rkcex}
{\sl The counterexample~{\rm(\ref{cex1})} due to Bahouri, G\'erard and Xu in\ccite{bgx} is given~by 
\beq
\label {soldataSH}u(t,Y,s) =  \int_{\R} {\rm e}^{i(s+4dt)\lam} {\rm e}^{-\lam |Y|^2} g(\lam) \, \lam^d d\lam \eeq
with~$g$    in~$\cD(]0,\infty[)$. Although $u_0$ does not belong to~$\cD(\H^{d})$, one can easily check that
  for any~$\delta>0$, 
 as soon as~$|s+4dt |> \delta |t|$ then    for any integer $N$ there is a positive constant~$C $   depending on~$N$  and~$u_0$ 
 such that \beq
\label {estcexdisp} |u(t,w)| \leq \frac {C   } {| s+4dt|^{N}} 
 \leq \frac {C  } {(\delta | t|)^{N}} \, \cdotp\eeq
 On the other hand Estimate~{\rm(\ref{estcexdisp})} fails, for any integer~$N \geq 1$, in the case when~$s=-4td$. This shows the sharpness of the bound on the constant~$\kappa$ appearing in Theorem~{\rm\ref{mainth1}}. More generally  it was established in\ccite{bdg} that  for any integer~$\ell$, denoting by~$L_\ell^{(d-1)}$    the  Laguerre polynomial of order~$\ell$ and type~$d-1$  (see for instance\ccite{Tricomi, gradshteyn, O}), then
 $$
    u^{(\ell)}(t,Y,s)   =\int_{\R} {\rm e}^{i(s+4t(2\ell+d))\lam}  {\rm e}^{ -|\lam||Y|^2 } L_\ell^{(d-1)}(2|\lam||Y|^2 ) g(\lam) \, \lam^d d\lam 
    $$  
    is a solution to  $(S_\H)$, and this solution   satisfies~{\rm(\ref{estcexdisp})}  when~$| s+ 4  (2\ell+d) t |>  \delta|t|, $  and~$|t|$ is large enough. 
 }
\end{remark}

 \medbreak

 As in the euclidean case outlined above,  the (local) dispersive estimate\refeq{dispSHloc}   stems from Young  inequalities\refeq{definyoungH} using an explicit formula of the type\refeq{kernelSR} for the Schr\"odinger kernel on~$\H^d$.  However, the study of the kernel~$S_t$  of the Schr\"odinger operator  on~$\H^{d}$  is more involved  than in the euclidean case,     because  on the one hand the Fourier transform   on~$\H^d$  is an intricate tool and on the other hand~$S_t$ does not enjoy a formulation  of type\refeq {kernelformula} globally on~$\H^d$.  In fact, as   will be seen in Section\refer{KS} (see Proposition\refer{Schkernelfourier}), one can compute~$S_t$  in the sense of distributions,  using  the Fourier-Heisenberg  analysis tools developed in\ccite{bcdh2}, and also  it turns out  (see Theorem\refer{STthdirac}) that   $ S_t$  concentrates  on quantized horizontal hyperplanes of~$\H^d$.  It follows  that  the explicit formula of the type\refeq{kernelSR} that we obtain here is only local. More precisely, our  result states as follows. Its sharpness is discussed in Paragraph~\ref{quant}.
  \begin{theorem}
\label{STth}
{\sl 
The kernel associated with the   free  Schr\"odinger equation~$(S_\H)$  reads for all~$t \neq 0$ 
\begin{equation}\label{SHkernel}S_t(Y,s)=\frac {1}  {(-4i \pi t)^{\frac Q 2} } \int_{\R}  \biggl(\frac {2 \tau}  {\sinh 2 \tau}\biggr)^{d} \, \exp\biggl( - \frac{\tau s}{2t}  - i \frac {|Y|^2\tau}  {2t \tanh  2 \tau} \biggr) d\tau\,  ,
\end{equation}
 provided that~$ |s|< 4d |t|$. }
\end{theorem}
   \begin{remark}
\label{rkSThd}
{\sl 
 Theorem\refer{STth}  highlights  the separate roles of the horizontal and vertical variables of~$\H^{d}$. Note that in\ccite{Muller}, D. M\"uller already emphasized the distinguished role of the horizontal variable in the study of  the Fourier restriction theorem on~$\H^{d}$. 
}
\end{remark}
 Even though  the  dispersive estimate\refeq{dispSHloc}   we   establish for the   Schr\"odinger   operator on~$\H^d$  is  only local,   we are  able to prove that the solutions of the Schr\"odinger  equation~$(S_\H)$ enjoy locally  Strichartz estimates in the spirit of\refeq{dispSTR}.  More precisely, we have  the following result.
\begin{theorem}
\label{STinvth}
{\sl
Under the notations of Theorem{\rm\refer{mainth1}}, given~$\kappa< \sqrt{4d}$ and $(p,q)$ belonging to  the admissible set
\beq
\label {admset}
\cA\eqdefa\Big\{(p,q)\, /\,  \frac2q+\frac{Q}p=\frac Q2 \with 2 \leq p \leq \infty\Big\} \, , \eeq
  there exists a  positive constant  $C(q,\kappa)$ such that,  for all $u_0 \in L^2(\H^d)$ supported in the ball~$B_\H(w_0, R_0)$, for some~$w_0 \in \H^d$,  the  solution to  the Cauchy problem $(S_\H)$  satisfies the following  local Strichartz estimate   
\beq
\label {dispSTlocl2bis}
\|u\|_{L^q(]-\infty, - C_\kappa R_0^2] \cup [C_\kappa R_0^2, +\infty[; L^p(B_\H(0, \kappa\sqrt{|t |})))} \leq
 C(q,\kappa)  \|u_0\|_{L^{2}(B_\H(w_0, R_0))}\, \virgp  \eeq
where $\ds C_\kappa=  {(\sqrt{4d}-\kappa)^{-2}}$.} 
\end{theorem}
Note that the Strichartz estimate{\rm\refeq{dispSTlocl2bis}} is invariant by scaling (through the scaling~$u (t,w)\mapsto u (\lambda^2    t,\delta_{\lambda}   w)   
$).
Let us underline that there is a duality between the size of the support of~$u_0$ and the time for which the Strichartz estimates holds. Indeed,   letting~$R_0$ 
go to zero,  we find that for an initial data concentrated around some~$w_0 \in \H^d$,    the Strichartz estimate is almost global in time.  Conversely, letting~$R_0$  go to infinity, the time from which{\rm\refeq{dispSTlocl2bis}} occurs is close to infinity.  Let us also emphasize  that
the counterexamples introduced in Remark{\rm\refer{mainth1rkcex}}  show   somehow the optimality of our result, since for these counterexamples a global  integrability  both with respect to $t$ and $s$ independently is excluded.

  \subsection{Refined study of the Schr\"odinger kernel  on~$\H^d$}\label {quant}  
   Theorem \refer{STth}
asserts that~$S_t$, for~$t \neq 0$,      is a decaying smooth function on the strip~$ |s|< 4d |t|$ (with a decay rate of order~$|t|^{-Q/2}$). One may wonder if $S_t$ ($t \neq 0$) which, according to Proposition\refer{Schkernelfourier},  belongs to~$\cS'(\H^d)$ can be identified with a function on the horizontal hyperplanes~$ s = \pm 4d |t|$. The  answer to this question is negative as asserted by the following result.
\begin{theorem}
\label{STthdirac}
{\sl 
With the previous notations,   for all $\pm w_\ell \eqdefa (0,  \pm 4(2\ell+d)|t|)$,  where $t \neq 0$ and~$\ell \in  \N$,   there exists an initial data $u_0^{\pm, \ell} \in \cS(\H^d)$ such that~$u^{\pm, \ell} (t, \cdot) \eqdefa\mathcal U(t)u_0^{\pm, \ell}$  satisfies 
\begin{equation}\label{SHkernellim} u^{\pm, \ell} (t, \pm w_\ell) = u_0^{\pm, \ell} (0)= \langle \delta_0, u_0^{\pm, \ell} \rangle_{\cS'(\H^d)\times \cS(\H^d)}\,  .
\end{equation}}
\end{theorem}
\begin{remark}
\label{speed}
{\sl  Actually, the above theorem can be easily generalized to any  element of the horizontal hyperplanes~$ s = \pm 4(2\ell+d)|t|$,~$t \neq 0$ and~$\ell \in  \N$, namely $(Y_0,  \pm 4(2\ell+d)|t|)$, where~$Y_0$ is some fixed element of $T^\star\R^d$. The  Cauchy data generating a solution which concentrates on   hyperplanes~$ s = \pm 4(2\ell+d)|t|$ are linked to the counterexamples introduced in Remark{\rm\refer{mainth1rkcex}}.
   }
\end{remark}
 
 The above result  shows the optimality of the bound $4d |t|$ in Theorem\refer{STth}. However, we are able to improve this bound when we restrict~$(S_\H)$
to some subspaces of Cauchy data as in the next statement.
\begin{theorem}
\label{better}
{\sl 
There exists an orthogonal decomposition of $L^2(\H^d)$
\begin{equation}\label{decorth} L^2(\H^d)=  \oplus_{m\in \N^d}  L_m^2(\H^d) \,  
\end{equation} 
such that the restriction~$S^{(\ell )}_t$  of $S_t$ to the subspace $\mathscr{V}_\ell (\H^d)\eqdefa \oplus_{|m|\geq \ell}  L_m^2(\H^d)$  is well defined as soon as~$ |s|< 4(2 \ell+ d )|t|$, and satisfies  for any positive  constant~$\kappa< \sqrt{4(d+2\ell)}$,  
  $$ \sup_{|s| \leq \kappa^2|t|} \sup_{Y \in T^*\R^d} \frac {1}  {|t|^{\frac Q 2} } |    S^{(\ell )}_t(Y,s) |\leq C(\ell,\kappa) \, .$$
 }
\end{theorem} 
 \begin{remark}
\label{decrk}
{\sl  Decomposition{\rm\refeq{decorth}} is strongly tied to the spectral representation of the sublaplacian~$-\D_\H$.   In order to give  a flavor of the above result, let us  point out that,    as we shall see in Section{\rm\refer{Fourier}},  the Fourier-Heisenberg transform exchanges~$-\D_\H$
 with  the harmonic oscillator. Then in some sense,{\rm\refeq{decorth}} consists in a  decomposition of $L^2(\H^d)$ along Hermite-type functions, via the Fourier-Heisenberg transform.}
\end{remark}

 \medbreak

\subsection{Main steps of the proof of the main results and layout of the paper}\label {steps}
Since the   Schr\"odinger equation   on~$\H^{d}$ is invariant under   left translations, one can assume without loss of generality in the proof of Theorem~\ref{mainth1}    that $w_0= 0$. By Young inequalities\refeq{definyoungH}, Theorem\refer{mainth1}    readily follows from Theorem\refer{STth} (reducing the assumption~$|s|<   {4d|t |}$ to the fact that~$\rho_\H (w) <  \sqrt{4d|t |}$). To prove Theorem\refer{mainth1}, we  are thus reduced to establishing Theorem\refer{STth}.  
Roughly speaking the proof of Theorem\refer{STth} is achieved in three steps. In the first step, using the   Fourier-Heisenberg  analysis on tempered distributions developed in\ccite{bcdh2} (see also Section\refer{tempered} in this paper),  
 we establish that   the kernel~$S_t$ of the Schr\"odinger operator on $\H^d$   belongs to~$\cS'(\H^d)$ (Proposition\refer{Schkernelfourier}). It is well-known since the   paper of B. Gaveau\ccite{B. Gaveau}, that the solution to the heat equation on~$\H ^d$  associated  with the sublaplacian writes for all $t > 0$ 
$$
 u(t, \cdot )=  \frac {1}  {t^\frac Q2 } u_0 \star (h \circ \delta_{\sqrt t }) \, ,$$
where~$ \delta_{\sqrt t }$ is the dilation operator defined in~(\ref{defdilation}) and~$h$ is the function in the  Schwartz class~$\cS(\H^d)$ given by
\begin{equation}\label{SHkernelex}
h(Y,s)  \eqdefa  \frac {1}  {(4\pi)^{\frac Q2} } \int_{\R}  \biggl(\frac {2 \tau}  {\sinh 2 \tau}\biggr)^{d} \, \exp\biggl( i \frac{\tau s}2  -  \frac { |Y|^2\tau}  {2\tanh  2 \tau} \biggr) d\tau\, .
 \end{equation}
Then the second step  is devoted to the  proof  of the fact that the fundamental solution of the heat equation on~$\H^d$ coming from     Fourier analysis on~$\H^d$ coincides with  the explicit formula\refeq{SHkernelex} established by B. Gaveau\ccite{B. Gaveau} (see Proposition\refer{Heatkernelfourier}). 
 This step uses Melher's formula, along with the  Fourier approach developed in\ccite{bcdh, bcdh2}. The  last step   concludes the proof following the general method of the euclidean case  via complex analysis,  described above (see Section\refer{localcomputkernel}). It is in this final step that the restriction~$ |s|< 4d |t|$ appears.
 
  \medskip      As usual, the proof of the  local Strichartz estimates stated in Theorem~\ref{STinvth}
 is straightforward from the local dispersive estimate\refeq{dispSHloc} thanks to standard functional analysis arguments.

  \medskip  Finally, the refined study of the Schr\"odinger kernel  on~$\H^d$ (through Theorems~\ref{STthdirac} and~\ref{better}) is derived by a combination of Fourier-Heisenberg tools and the spectral analysis of the harmonic oscillator.

 \medskip
 
 Let us describe the organization of the paper.
Section\refer{Fourier}  is dedicated to  a brief description of  the Fourier transform~$\cF_\H$ on~$\H^d$ and the space of frequencies~$\wh\H^d$, as well as the  extension of $\cF_\H$  to tempered distributions -- which is at the heart of the matter in this paper.  In Section\refer{KH}, we recover  the explicit formula of   the heat kernel  on~$\H^d$  established by B. Gaveau in\ccite{B. Gaveau},  using     Fourier analysis on $\H^{d}$. In Section\refer{KS},  we investigate  the kernel  of the Schr\"odinger operator on~$\H^d$ and prove Theorem\refer{STth}, while    Section\refer{proofmain1} is devoted to the proof of Theorem\refer{mainth1}    thanks to Theorem~\ref{STth}.  Then, we establish  the local Strichartz estimates (Theorem~\ref{STinvth}).   In Section\refer{kerneldirac}, we undertake a refined study of~$S_t$ and establish Theorems\refer{STthdirac} and\refer{better} making use of  the Fourier-Heisenberg approach developed in\ccite{bcdh, bcdh2}.  
\medskip  

To avoid heaviness, all along this article~$C$ will denote  a  positive  constant   which may vary from line to line.    We also use $A\lesssim B$  to
denote an estimate of the form~$A\leq C B$.

\medskip

{\bf Acknowledgements. } $ $The authors thank  Nicolas Lerner  and Jacques Faraut very warmly for  their help and input  concerning  
the proof of Theorem~\ref{better}.

\section{Fourier analysis on $\H^d$}\label{Fourier}

\subsection{The Fourier transform on $\H^d$}\label {fouriertransf}
The   Fourier transform on $\H^d$ is defined using irreducible unitary representations of $\H^d$.  It is thus not 
a complex-valued function on some ``frequency space" as in the euclidean case, but a family of bounded operators on~$L^2(\R^d)$ (see for instance~\cite{astengo2, corwingreenleaf, farautharzallah, FF, stein2, taylor1, thangavelu} for further details). 
Recently, in\ccite{bcdh} and\ccite{bcdh2}  the authors  introduced an equivalent, intrinsic definition of the Fourier transform on~$\H^d$ in terms of functions acting on a frequency set denoted~$\wt\H^d\eqdefa\N^{2d}\times \R\setminus\{0\}$. 
More precisely, denoting the elements of this set by~$\wh w\eqdefa(n,m,\lambda)$, the Fourier transform of  an integrable function on~$\H^d$ is defined     in the following way:
 \beq
\label {defF}
\forall \wh w \in \wt\H^d \, , \quad \cF_\H f(\wh w) \eqdefa \int_{\H^d}  \overline{{\rm e}^{is\lam} \cW(\wh w,Y)}\, f(Y,s) \,dY\,ds\,,
 \eeq
 with $\cW$ the Wigner transform of the (renormalized) Hermite functions
 \beq
\label {def}
 \cW(\wh w,Y)
\eqdefa\int_{\R^d} {\rm e}^{2i\lam\langle \eta,z\rangle} H_{n,\lam}(y+z) H_{m,\lam} (-y+z)\,dz\,, \quad H_{m,\lam} (x)\eqdefa |\lam|^{\frac d 4} H_m(|\lam|^{\frac 12} x)\,   ,
 \eeq
 with~$ \suite H m {\N^d}$ the Hermite orthonormal basis of~$L^2(\R^d)$ given by the eigenfunctions of the harmonic oscillator:
$$
 -(\D -|x|^2) H_m= (2|m|+d) H_m\, .
$$
We recall that
 \beq
\label {hermite}
H_m(x) \eqdefa  \Bigl(\frac 1 {2^{|m|} m!}\Bigr) ^{\frac 12} \prod_{j=1}^d  \big(-\partial_j H_0(x)+ x_jH_0(x)\big)^{m_j} \, ,
\eeq
with $H_0(x)\eqdefa \pi^{-\frac d 4} {\rm e}^{-\frac {|x|^2} 2}$,  $m!\eqdefa m_1!\dotsm m_d!\,$ and $\,|m|\eqdefa m_1+\dots+m_d.$ 

  In\ccite{bcdh}, the authors show that the completion of the set~$\wt \H^d$  for the    distance  
$$
\wh d(\wh w,\wh w') \eqdefa  \bigl|\lam(n+m)-\lam'(n'+m')\bigr|_{\ell^1(\N^d)} + \bigl |(n-m)-(n'-m')|_{\ell^1(\N^d)}+d|\lam-\lam'|  
$$
 is   the set
$$
\wh \H^d\eqdefa\wt\H^d
 \cup \wh \H^d_0 \with \wh \H^d_0 \eqdefa {\R_{\mp}^d}\times \Z^d
\andf
{\R_{\mp}^d}\eqdefa  (\R_-)^d\cup (\R_+)^d\,.
$$
In this setting, the  classical 
     statements of Fourier analysis 
hold   in a similar way to the euclidean case.  In particular,  the inversion and Fourier-Plancherel formulae read:
\beq
\label {inverseFourierH}f(w) = \frac {2^{d-1}}  {\pi^{d+1} }   \int_{\wt \H^d} 
{\rm e}^{is\lam} \cW(\wh w, Y)\cF_\H f(\wh w) \, d\wh w 
 \eeq
and
\beq
\label {FourierplancherelH}
(\cF_\H f|\cF_\H g)_{L^2(\wt \H^d)}  = \frac {\pi^{d+1}} {2^{d-1}} (f|g)_{L^2(\H^d)}\, ,
 \eeq
where   the measure~$d\wh w$ is defined in the following way\footnote{As shown in\ccite{bcdh2},   the measure~$d\wh w$ can be extended by $0$  on  $\wh \H^d_0$.}: for any function~$\theta$ on~$\wt \H^d$,
$$
\int_{\wt \H^d} \theta (\wh w)\,d\wh w \eqdefa \int_{\R} \sum_{(n,m)\in \N^{2d} }\theta(n,m,\lam) |\lam|^d\,d\lam\,.
$$
Straightforward computations give
    \beq
\label {actionlap}
\cF_\H( -\Delta_\H f) (\wh w) =4| \lam | (2|m|+d)\cF_\H( f)(\wh w)\,.
 \eeq
According to\refeq{inverseFourierH}-\eqref{FourierplancherelH}, one can easily check that 
\beq \label{defL2m} L^2(\H^d)= \displaystyle \oplus_{m\in \N^d}  L_m^2(\H^d) \,,
\eeq
in the following way:  any function $f  \in L^2(\H^d)$  can be split as
 \beq
\label {l2m}
f=\sum_{m\in \N^d} f_m\, \,\mbox{with} \,\, f_m(Y,s)= \frac {2^{d-1}}  {\pi^{d+1} }  \sum_{n\in \N^{d} } \int_{\R} 
{\rm e}^{is\lam} \cW((n, m, \lam), Y)\cF_\H f(n, m, \lam) |\lam|^d\,d\lam, \eeq
and
$$\|f\|^2_{L^2(\H^d)}= \sum_{m\in \N^d} \|f_m\|^2_{L^2(\H^d)}\,.$$
Let us also note that if~$f$ and~$g$ are two functions of~$L^1(\H^d)$ then for any~$\wh w=(n,m,\lambda)$ in~$\wt\H^d,$ there holds
\beq
\label {newFourierconvoleq1}
 \cFH (f\star g) (\wh w)   = ( \cF_\H f \cdot \cF_\H g)(\wh w)\eqdefa \sum_{p\in \N^{d}} \cF_\H f(n,p,\lam)\cF_\H g(p,m,\lam)\,.
 \eeq
As we shall see,  the heat and Schr\"odinger  kernels  on $\H ^d$  are  radial, in the sense that they are  invariant under the action of  the unitary group of~$T^\star\R^d$. In addition, being functions of~$-\Delta_\H$  they are   even. In fact, the Fourier transform of radial functions turns out to be simpler than in the general case: if~$f$ is a  radial function  in $L^1(\H^d)$, then for any~$(n,m,\lam) \in \wt \H^d$,
\begin{equation}\label {Fradial}  \cFH (f) (n, m, \lam)=  \cFH (f) (n, m, \lam)\delta_{n,m}= \cFH (f) (|n|, |n|, \lam) \delta_{n,m} \,,
\end{equation} 
with, for all $\ell \in \N$, 
\begin{equation}\label {Fradial2} 
\cFH (f) (\ell, \ell, \lam) =  \begin{pmatrix} \ell +d-1 \\ \ell \end{pmatrix}^{-1}\int_{\H^d}   {{\rm e}^{-is\lam}} {\rm e}^{ -|\lam||Y|^2 } L_\ell^{(d-1)}(2|\lam||Y|^2 )  f(Y,s) \,dY\,ds\, ,
\end{equation} 
where  $L_\ell^{(d-1)}$ stands for  the  Laguerre polynomial\footnote{The interested reader can consult for instance\ccite{askey, beals, Tricomi} and the references therein.} of order $\ell$ and type $d-1$. 
  
\medskip  Obviously the inversion   formula  writes in that case
\begin{equation}\label {Fradial3} 
f(w) = \frac {2^{d-1}}  {\pi^{d+1} }  \sum_{\ell \in \N} \int_{\R} 
{\rm e}^{is\lam} \wt \cW(\ell, \lam, Y)\cFH (f) (\ell, \ell, \lam)\, |\lam|^d\,d\lam\, ,
\end{equation} 
where    \begin{equation}
\label{princident} \wt \cW(\ell,\lam,Y) \eqdefa \sumetage {n\in\N^d} {|n|= \ell} \cW(n,n,\lam,Y ) ={\rm e}^{- |\lam||Y|^2 } L_\ell^{(d-1)} ( 2  |\lam|  |Y|^2)\, .
 \end{equation}
\subsection{The Fourier transform  on~$\cS'(\H^{d})$}\label {tempered}
The new approach of the   Fourier-Heisenberg  transform developed in\ccite{bcdh} enabled the authors in\ccite{bcdh2}  to extend $\cF_\H$ to $\cS'(\H^d)$,  the set of tempered distributions: note that since the Schwartz class~$\cS(\H^d)$ coincides with $\cS(\R^{2d+1})$ then similarly~$\cS'(\H^{d})$  is nothing else than~$\cS'(\R^{2d+1})$. Roughly speaking, the first step to achieve this  extension consists in characterizing~$\cS(\wh \H^d)$,    the range of $\cS(\H^d)$, by $\cF_\H$.  It will be useful to recall in the following   that according to~\cite{bcdh2}, the space~$\cS(\wh \H^d)$ can be  equipped  with semi-norms~$\|\cdot\|_{N,\cS(\wh \H^d)}$ and that in particular, for all~$\theta \in \cS(\wh \H^d)$ and~$N \in  \N$,  there exists~$C_N$  such that for all~$\hat w=(n,m, \lam) \in \wt  \H^d$
 \begin{equation}
\label{semipart} 
|\theta (\hat w)| \leq C_N(1+  4|\lam|(2|m|+d))^{-N} \|\theta \|_{N,\cS(\wh\H^d)}\, .
 \end{equation}
 We refer to\ccite{bcdh2}  for    the definition of~$\cS(\wh \H^d)$ and further details. Then the result follows  by duality, as in the euclidean case, once   shown that the Fourier transform~$\cF_\H$ is a bicontinuous isomorphism between the spaces~$\cS(\H^{d})$ and~$\cS(\wh \H^d)$\footnote{Note that a  first attempt in the  description of the range of~$\cS(\H^{d})$ by the Fourier-Heisenberg transform goes back to the pioneering works by D. Geller in\ccite{geller,geller2}, where asymptotic series are used. 
One can also consult the paper of F. Astengo, B.  Di Blasio and F. Ricci\ccite{astengo2}.}.

 The map~$\cF_{\H}$ can thus be continuously extended from~$\cS'(\H^d)$ into~$\cS'(\wh\H^d)$ in the following way:
\beq
\label {defS'}
\cF_{\H} : \left \{
\begin{array}{ccl}
\cS'(\H^d) & \longrightarrow & \cS'(\wh \H ^d) \\
T & \longmapsto & \Bigl[\theta \mapsto  \langle \cF_{\H}T, \theta\rangle_{\cS'(\H^d)\times \cS(\H^d)}= \langle T,{}^t\!\cF_{\H}\theta\rangle_{\cS'(\H^d)\times \cS(\H^d)}\Bigr]\, ,
\end{array}
\right.
\eeq
where, according to\refeq{inverseFourierH},  \beq
\label {linetFF-1}
 {}^t\cF_{\H}\theta(y,\eta,s) \eqdefa \frac {\pi^{d+1}} {2^{d-1}} (\cF_{\H}^{-1}\theta)(y,-\eta,-s) \, .
\eeq
In particular one can compute the Fourier transform of the Dirac mass:
$$
\cF_{\H}(\delta_0)   =  {\mathbf 1} _{\{(n,m,\lam)\,/\ n=m\}}\,, 
$$
that is to say, for any  $\theta$ in~$\cS(\wh\H^d)$  
\beq
\label {DiracF}
\langle \cF_{\H}(\d_0),\theta\rangle _{\cS'(\wh \H ^d)\times\cS(\wh \H ^d)}  = \sum_{n\in \N^d}\int_{\R} \theta (n,n,\lam)|\lam|^dd\lam \, . 
\eeq

\medbreak

It will be useful later on to notice that~$\cS'(\wh\H^d)$  contains (after suitable identification) all functions with moderate growth, that are defined as the locally integrable functions~$\theta$ on~$\wh \H^d $ such that
  for some large enough integer~$N,$ the map
 \beq
\label {moderate}(n,m,\lambda)\longmapsto \big(1+|\lam|(n+m|+d)+|n-m|\big)^{-N}\theta(n,m,\lam)\eeq belongs to~$L^\infty(\wt \H^d)$. It is proved in\ccite{bcdh2} that any such function can be identified  with a tempered distribution on~$\wh\H^d.$  
\medbreak

Let us end this   introduction on Fourier analysis on the Heisenberg group by recalling   that if $T$ is a tempered distribution on $\H^d$,  then  for all~$f$   in~$\cS(\H^d)$  and all~$w$ in~$\H^d,$  
$$
(T\star f)(w)=\langle T,\check f\circ\tau_{w^{-1}}\rangle_{\cS'(\H^d)\times\cS( \H^d)} \, ,
$$
and   
\begin{equation}\label{eq:defconvright}
(f\star T)(w)=\langle T,\check f\circ\tau^{\rm r}_{w}\rangle_{\cS'(\H^d)\times\cS( \H^d)} \, ,
\end{equation} 
where
\beq
\label {defchech}
  \check f (w)\eqdefa f (w^{-1}) \,, 
  \eeq
and~$\tau_w$   denotes the left translation operator by $w$ defined in~(\ref{definlefttranslate}) while~$\tau^{\rm r}_w$    is the right translation operator by $w$ defined by
 \begin{equation}\label{eq:defright}
 \tau^{\rm r}_w (w')\eqdefa w'\cdot w\, .
 \end{equation}

\medbreak

 \section {On the kernel  of the heat operator on $\H^{d}$}\label{KH}
A striking  consequence of   Fourier analysis on $\H^{d}$  developed in\ccite{bcdh2} is  that 
it provides another  proof  of the fact that the fundamental solution of the heat equation on~$\H^d$  coincides with  the explicit formula established by B. Gaveau in\ccite{B. Gaveau}. 
 First note  the following representation coming from Fourier-Heisenberg analysis.\begin{proposition}
\label {Heatkernelfourier}
{\sl If $u$ denotes the solution to the   free  heat equation  on $\H^d$ 
$$
(H_\H)\qquad \left\{
\begin{array}{rcl}
\ds \partial_t u -  \D_\H u & =  &0\\
{ u}{}_{|t=0} &= & u_0\, ,
\end{array}
\right.
$$
where $u_0$ is a given integrable function on $\H^d,$   then for all $t > 0$  there holds
$$
 u(t, \cdot )= u_0 \star h_t\, ,
$$
where   $h_t$ is   defined by   
$$h_t(Y,s)\eqdefa \frac {2^{d-1}}  {\pi^{d+1} }  \sum_{n\in \N^{d} } \int_{\R} {\rm e}^{is\lam} \cW \big((n,n,\lambda), Y\big)  {\rm e}^{- 4 t |\lam| (2 |n|+d)}   |\lam|^d d \lam\, .
$$
}
\end{proposition}
\begin{proof}
Applying  $\cF_\H$ to the heat equation and taking advantage of\refeq{actionlap}, we get 
 for all~$(n,m,\lambda)$ in~$\wt \H^d,$
$$
\left\{
\begin{array}{rcl}
\ds  \frac  {d \, \wh u_\H} {dt}  (t, n,m, \lam)&\! =\!  &- 4|\lam|(2|m|+d) \wh u_\H (t,n,m, \lam)\\
{ \wh u_\H}{}_{|t=0} &\!=\! &   \cF_\H  {u_0}\, .
\end{array}
\right.
$$
By time integration, this implies  that  for all~$(n,m,\lambda)$ in~$\wt \H^d,$  
\beq
\label  {noyaychaleurb}
 \wh u_\H (t, n,m, \lam)=   {\rm e}^{- 4 t |\lam|(2|m|+d)} \cF_\H  {u_0} (n,m, \lam) \, .
\eeq
According to\refeq{newFourierconvoleq1}, we deduce that 
$$  \wh u_\H (t, n,m, \lam) = (  \cF_\H  {u_0}\cdot \theta_t)(n,m,\lam)\with \theta_t(n,m,\lam) \eqdefa {\rm e}^{- 4 t |\lam|(2|n|+d)} \d_{n,m}  \, ,$$
where $\d_{n,m}$ denotes the Kronecker symbol, which implies that  
$$u(t, \cdot )= u_0 \star h_t \with  (\cF_\H  {h_t})(n,m, \lam)\eqdefa {\rm e}^{- 4 t |\lam|(2|n|+d)} \d_{n,m} \, . $$
This concludes the proof of the proposition thanks to   
 the inversion  formula~(\ref{inverseFourierH}). \end{proof}  

\medbreak

\begin{remark}
\label {Heatkernelrk}
{\sl Performing the change of variable $  t \lam \mapsto \lam$  in the heat kernel given by Proposition~{\rm\ref{Heatkernelfourier}} readily implies that for all $t > 0$ 
$$h_t(Y,s)= \frac{1}{t^{\frac Q 2}} h\biggl(\frac{Y}{\sqrt{t}} \,, \frac s t\biggr)\, \virgp$$
with 
\beq 
\label{defhF}
h(Y,s)\eqdefa\frac {2^{d-1}}  {\pi^{d+1} }  \sum_{m\in \N^{d} } \int_{\R} 
{\rm e}^{is\lam} \cW\big((n,n,\lambda), Y\big) {\rm e}^{- 4 |\lam| (2 |m|+d)} \delta_{n,m} 
 |\lam|^d d\lam \,.
 \eeq}
\end{remark}

\medbreak
The following remarkable result due to B. Gaveau (\cite{B. Gaveau}) asserts that   the heat  operator on~$\H^d$ has a  convolution kernel in~$\cS(\H^{d})$. 
 \begin{theorem}  [\cite{B. Gaveau}]
\label {Heatkernel}
{\sl There exists a function $h$ in~$\cS(\H^{d})$  such that for any~$u_0 \in L^1(\H^d)$, the  solution to~$(H_\H)$
writes for all $t > 0$ 
$$
 u(t, \cdot )= u_0 \star h_t\, ,
 $$
where   $h_t$ is defined by 
 \beq
\label  {noyaychaleur}
h_t(Y,s)= \frac{1}{t^{\frac Q 2}} h\biggl(\frac{Y}{\sqrt{t}} \, , \frac s t\biggr) 
\eeq
and the function $h$ is given by the  formula
\beq 
\label{defhG}
h(Y,s)
\eqdefa  \frac {1}  {(4\pi )^{\frac Q 2} } \int_{\R}  \biggl(\frac {2 \tau}  {\sinh 2 \tau}\biggr)^{d} \, \exp\biggl( i \frac{\tau s}{2}   -  \frac {|Y |^2\tau}  {2  \tanh  2 \tau} \biggr) d\tau  \,  .
\eeq
}
\end{theorem}
\begin{proof}
Our purpose here is to establish that  the   formula coming from Fourier analysis given by~(\ref{defhF}), namely
$$
h(y,\eta,s)
= \frac {2^{d-1}}  {\pi^{d+1} }  \sum_{m\in \N^{d} } \int_{\R\times \R^{d}} 
{\rm e}^{is\lam+2i\lam \langle \eta,z\rangle} {\rm e}^{- 4 |\lam| (2 |m|+d)}
H_{m,\lam} (y+z)
H_{m,\lam} (-y+z) \,dz |\lam|^d d\lam
$$
coincides with  the   explicit expression of the heat kernel~(\ref{defhG})
   given by Theorem\refer{Heatkernel}.
The proof   relies on   Melher's formula  (see\ccite{Feynman})
 \begin{equation}\label{mehler} \sum_{m \in \N} P_m(x)\, P_m(\wt x) \, r^m = \frac1{\sqrt{1-r^2}} \,  {\rm exp}\Big( \frac {2 x \wt x r-(x^2+\wt x^2) r^2} {1-r^2} \Big) \,,\end{equation}
that holds true  for all~$x, \wt x$ in $\R$ and $r$ in $]-1,1[,$
 where $P_m$ denotes   the Hermite polynomial of order $m$ defined by 
 $$  P_m (x)\eqdefa  \pi^{\frac 1 4}H_{m} (x)  {e}^{\frac {|x|^2} 2} \, , $$
 with $H_m$    the Hermite function  introduced in\refeq{hermite}.
 \medbreak
 To this end, we shall   use      the following  lemma.
  \begin{lemma}\label {lemmaMehler}
{\sl 
Under the above notations, there holds for all~$(y,z)\in \R^2$ and all positive real numbers~$\lam$ and $t,$
$$
\sum_{m \in\N} {\rm e}^{-2m t\lam} H_{m, \lam}(z-y) \, H_{m,\lam }(z+y) = \frac1{\pi^{\frac 1 2}} \, \Big(\frac{  {\lam}}{1-{\rm e}^{- 4  t\lam  }}  \Big)^\frac12 
 \exp \Big( - \lam z^2 \tanh (t\lam)
 - \frac{\lam y^2}{\tanh (t\lam)} \Big)
\, .
$$}
\end{lemma}
\begin{proof}
Applying the Mehler formula\refeq{mehler}  to the rescaled Hermite functions\refeq{hermite} 
 yields
$$
\displaylines{\quad
\sum_{m \in \N}  {\rm e}^{- 2t m \lam   }  H_{m, \lam  }(z-y) \, H_{m, \lam}(z+y)  =   \frac1{\pi^{\frac 1 2}} \, {\rm e}^ {- \lam (z^2+y^2)} \Big(\frac{  {\lam}}{1-{\rm e}^{- 4t \lam   }}  \Big)^\frac12 
\hfill\cr\hfill
  \times \exp  \Big(\frac{1}{1-{\rm e}^{- 4t  \lam   }}\Big(
2 \lam (z^2 - y^2){\rm e}^{- 2t  \lam   } - 2\lam (z^2 + y^2){\rm e}^{- 4t  \lam   }
\Big)\Big)\, .}
$$
The result follows from the following identities:
$$
\displaylines{\quad
- \lam  (z^2+y^2)  + \frac{1}{1-{\rm e}^{- 4t  \lam   }}\Big(
2 \lam (z^2 - y^2){\rm e}^{- 2t  \lam   } - 2\lam (z^2 + y^2){\rm e}^{- 4t  \lam   }
\Big)\hfill\cr\hfill = -\frac{\lam}{1-{\rm e}^{- 4t  \lam   }} \big(
z^2 (1-{\rm e}^{-2t \lam})^2 + y^2 (1+{\rm e}^{-2t \lam})^2 
\big)\, , }
$$
$$
 \frac{ (1-{\rm e}^{-2t\lam})^2}{1-{\rm e}^{- 4t  \lam   }} =
 \frac{ ({\rm e}^{t \lam}-{\rm e}^{-t\lam})^2}{{\rm e}^{2t\lam}-{\rm e}^{- 2t  \lam  }}  =
  \frac{ {\rm e}^{t \lam}-{\rm e}^{-t\lam}}{{\rm e}^{t\lam}+{\rm e}^{- t  \lam  }}
 = \tanh (t\lam)  \, , 
$$
and
$$
 \frac{ (1+{\rm e}^{-2t\lam})^2}{1-{\rm e}^{- 4t  \lam   }} =
 \frac{ ({\rm e}^{t \lam}+{e}^{-t\lam})^2}{{\rm e}^{2t\lam}-{\rm e}^{- 2t  \lam  }}   =
  \frac{{\rm e}^{t \lam}+{\rm e}^{-t\lam}}{{\rm e}^{t\lam}-{\rm e}^{- t  \lam  }} =\frac1{ \tanh (t\lam)} \, \cdotp
$$
The lemma is proved.
\end{proof}

\medbreak
Let us return to the proof of Theorem~\ref{Heatkernel}.
In order to establish that the two formulae~(\ref{defhF}) and~(\ref{defhG}) coincide, it suffices to consider the case when~$d=1$:~(\ref{defhF}) becomes
\begin{equation} \label{formulausef}
h(y,\eta,s)= \frac {1}  {\pi^{2} }  \sum_{m\in \N } \int_{\R^2} 
{\rm e}^{is\lam+2i\lam\eta z} {\rm e}^{- 4 |\lam| (2 m+1)}
H_{m,\lam} (y+z)
H_{m,\lam} (-y+z) \,dz |\lam| d\lam\,.
\end{equation}
Then   applying  Lemma\refer{lemmaMehler} with $t=4|\lambda|$,  we find that
$$
\begin{aligned}
h(y,\eta,s)&= \frac {1}  {\pi^{\frac 5 2} } \int_{\R^2} \!
{\rm e}^{is\lam +2i\lambda\eta z- 4 |\lam| }  \biggl(\frac{|\lambda|}{1-{\rm e}^{-16|\lambda|}}\biggr)^{\frac12}\!\\
&\qquad \times
\exp\biggl(-|\lambda| z^2\tanh (4 |\lam|) -\frac{|\lambda| y^2}{\tanh(4|\lambda|)}\biggr)
|\lambda|  d\lam dz\, .
\end{aligned}
$$
Performing  the change of variables~$|\lam|^{\frac12} \, z\mapsto   z$, this gives rise to 
$$
h(y,\eta,s)= \frac {1}  {\pi^{\frac 5 2} } \int_{\R} 
{\rm e}^{is\lam - 4 |\lam| - \frac{|\lam|  y^2}{\tanh (4 |\lam| )}} {\mathcal F}\big({\rm e}^{-
\tanh (4 |\lam|  )\, |\cdot|^2}
\big)\big(2|\lam|^{\frac12} \eta\big) \,\Big(\frac{  {|\lam|}}{1-{\rm e}^{- 16|\lam|  }}  \Big)^\frac12 |\lam|^{\frac12} d\lam\,.
$$
Since $${\mathcal F}\big({e}^{-
\tanh (4 |\lam|  )\,  |\cdot|^2}
\big)\big(2|\lam|^{\frac12}\eta\big)= \sqrt{ \frac \pi {
\tanh (4 |\lam|  )}} \, {\rm e}^{-\frac {|\lam| \eta^2}{\tanh (4 |\lam|  )}}\, ,$$
 we discover that
 $$
h(y,\eta,s)= \frac {1}  {\pi^{2} } \int_{\R} 
{\rm e}^{is\lam - 4 |\lam| - \frac{|\lam|  (y^2+\eta^2)}{\tanh (4 |\lam| )}} 
\Big(\tanh (4 |\lam|  ) (1-{\rm e}^{- 16 |\lam|  })\Big)^{-\frac12}
 \,  |\lam|  \,  d\lam\,.
$$
But
$$
\tanh (4  |\lam|  ) (1-{\rm e}^{-16  |\lam| }) 
= 4 \, {\rm e}^{-8 |\lam| } \sinh^2(4 |\lam| )\, ,
$$
which ends  the proof of Theorem~\ref{Heatkernel}.
\end{proof}

\medbreak
 
 \section {On the kernel  of the Schr\"odinger operator on $\H^d$}\label{KS}
 \subsection {Representation of  the free Schr\"odinger equation}\label{repsolS} Contrary to the heat equation (and as in the euclidean case recalled in the introduction), the kernel  of the Schr\"odinger operator does not belong to the  Schwartz class~$\cS(\H^d)$.  Nevertheless, one can  solve explicitly the Schr\"odinger  equation  $(S_\H)$
  by means of the   Fourier-Heisenberg  transform introduced in Section\refer{Fourier}, in the following way.

  \begin{proposition}
\label {Schkernelfourier}
{\sl The solution to the   free  Schr\"odinger equation
$$
(S_\H)\qquad \left\{
\begin{array}{c}
i\partial_t u -\D_\H u = 0\\
u_{|t=0} = u_0\,,
\end{array}
\right. 
$$
 reads for all $t \neq 0$  and all~$u_0 \in \cS(\H^d)$
\begin{equation}\label{eq:schconv}
 u(t, \cdot )= u_0 \star S_t\,,
 \end{equation}
where   $S_t$ denotes  the tempered distribution on $\H^d$ defined for all $\varphi$ in $\cS(\H^d)$ by  
$$\langle S_t, \varphi\rangle_{\cS'(\H^d)\times \cS(\H^d)} = \langle {\rm e}^{4 i t |\lam|(2|n|+d)} \d_{n,m},\theta\rangle_{\cS'(\wh\H^d)\times\cS(\wh \H^d)}  \, ,$$  
with $\varphi= {}^t\!\cF_{\H}\theta$, according to  Notation~{\rm(\ref{linetFF-1})}.  
}
\end{proposition}
\begin{proof}
Arguing as for the proof of Proposition\refer{Heatkernelfourier}, we start by applying  $\cF_\H$ to $(S_\H)$, which thanks to\refeq{actionlap} implies that 
$$
\left\{
\begin{array}{rcl}
\ds i  \frac  {d \, \wh u_\H} {dt}  (t, n,m, \lam)&\! =\!  &- 4|\lam|(2|m|+d) \wh u_\H (t,n,m, \lam)\\
{ \wh u_\H}{}_{|t=0} &\!=\! &   \cF_\H  {u_0}\, ,
\end{array}
\right.
$$
and leads by  integration to \beq
\label  {noyayschb}
 \wh u_\H (t, n,m, \lam)=   {\rm e}^{4 i t |\lam|(2|m|+d)} \cF_\H  {u_0} (n,m, \lam) \, ,
\eeq for all~$(n,m,\lambda)$ in~$\wt \H^d$. Then taking advantage of\refeq{newFourierconvoleq1}, we find that 
$$  \wh u_\H (t, n,m, \lam) = (  \cF_\H  {u_0}\cdot \Theta_t)(n,m,\lam)\with \Theta_t(n,m,\lam) \eqdefa {\rm e}^{4 i t |\lam|(2|n|+d)} \d_{n,m}  \, .$$
One can easily check that $\Theta_t$ is a function with moderate growth in the sense of\refeq{moderate},   and thus as it was proved in\ccite{bcdh2}, it is a tempered distribution on $\wh\H^d$. This ensures that the  Schr\"odinger kernel  $S_t$ belongs to $\cS'(\H^d)$. 

\medskip Finally combining\refeq{linetFF-1} together with\refeq{eq:defconvright},   we readily gather that   for all~$u_0$   in~$\cS(\H^d)$  and all~$w$ in~$\H^d,$
\begin{equation} \begin{aligned}\label{solSHkernel}  (u_0 \star S_t)(w) = \langle S_t, \check u_0 \circ\tau^{\rm r}_{w}\rangle_{\cS'(\H^d)\times \cS(\H^d)} & = \langle {\rm e}^{4 i t |\lam|(2|n|+d)} \d_{n,m},\theta\rangle_{\cS'(\wh\H^d)\times\cS(\wh \H^d)}  
\\ & = \sum_{n\in \N^d}\int_{\R}{\rm e}^{4 i t |\lam|(2|n|+d)}   \theta (n,n,\lam)|\lam|^dd\lam \, , \end{aligned}\end{equation}
 where  $\check u_0 \circ\tau^{\rm r}_{w}= {}^t\!\cF_{\H}\theta$, which completes the proof of the proposition.    \end{proof}  
\subsection {Computation of  the  Schr\"odinger kernel on Heisenberg strips}\label {localcomputkernel}
Our  goal  now   is to  establish Theorem\refer{STth}.  As already mentioned,  the proof of Formula\refeq{SHkernel} goes along the same lines as the euclidean proof, though more involved.   Thanks to Theorem\refer{Heatkernel},   the solution to the heat equation~$(H_\H)$ writes 
$$
 f(t, \cdot )= f_0 \star h_t\,,
 $$ 
 where   $h_t$ is given  for all $t > 0$ by
 $$
  \begin{aligned} h_t(Y,s)&= \frac {1}  {(4\pi t)^{\frac Q 2} } \int_{\R}  \biggl(\frac {2 \tau}  {\sinh 2 \tau}\biggr)^{d} \, \exp\biggl( i \frac{\tau s}{2t}   -  \frac {|Y |^2\tau}  {2 t \tanh  2 \tau} \biggr) d\tau\\ &=\frac {2^{d-1}}  {\pi^{d+1} }  \sum_{m\in \N^{d} } \int_{\R} 
{\rm e}^{i s \lam} {\rm e}^{- 4 t |\lam| (2 |m|+d)}\cW\big((m,m,\lam),Y\big)
|\lam|^d d\lam \,.
\end{aligned} 
$$
To achieve our goal, the first step   consists in observing  that  the 
maps
$$z    \longmapsto   H^1_z(Y,s)  \andf z    \longmapsto   H^2_z(Y,s) $$
with 
\begin{equation}\label{defH1}
\ds H^1_z(Y,s) \eqdefa \ds \frac {2^{d-1}}  {\pi^{d+1} }  \sum_{m\in \N^{d} } \int_{\R} 
{\rm e}^{i s \lam} {\rm e}^{- 4 z |\lam| (2 |m|+d)} \cW\big((m,m,\lam),Y\big) |\lam|^d d\lam 
 \end{equation}
and \begin{equation}\label{defH2} H^2_z(Y,s) \eqdefa \ds      \frac {1}  {(4\pi z)^{\frac Q 2} } \int_{\R}  \biggl(\frac {2 \tau}  {\sinh 2 \tau}\biggr)^{d} \, \exp\biggl( i \frac{\tau s}{2z}  -  \frac { |Y |^2 \tau}  {2 z \tanh  2 \tau} \biggr) d \tau
\end{equation}
are,  for   all $(Y,s)$ in $\H^d$, holomorphic on a suitable domain of $\C$.

\bigskip Actually  on the one hand, performing the change of variables $\beta= \lam(2|m|+d)$ in each integral of the right-hand side of\refeq{defH1}, we get
 $$
 H^1_z(Y,s) = \ds \frac {2^{d-1}}  {\pi^{d+1} }  \sum_{m\in \N^{d} }  \frac 1 {(2|m|+d)^{d+1}}\int_{\R} 
{\rm e}^{i s \frac \beta {2|m|+d}} {\rm e}^{- 4 z |\beta| } \cW\Big(\Big(m,m,\frac \beta {2|m|+d}\Big),Y\Big) |\beta|^d d\beta \, ,
$$ where obviously in each term of the above identity  the integrated function is holomorphic on $\C$. Moreover, using the fact that the modulus of   the Wigner transform of the   Hermite functions is bounded by $1$, we obtain  for all $z$ in $\C$ satisfying ${\rm Re}(z) \geq a >0 $
\begin{multline*} 
\int_{\R} \Big|{\rm e}^{i s \frac \beta {2|m|+d}} {\rm e}^{- 4 z |\beta| }
 \cW\Big(\Big(m,m,\frac \beta {2|m|+d}\Big),Y\Big) \Big| \,  |\beta|^d d\beta \leq  \int_{\R} {\rm e}^{- 4 a |\beta| } |\beta|^{d} d\beta < \infty \andf \\ 
 \int_{\R} \Big|\partial_z\Big({\rm e}^{i s \frac \beta {2|m|+d}} {\rm e}^{- 4 z |\beta| } \cW\Big(\Big(m,m,\frac \beta {2|m|+d}\Big),Y\Big) \Big)\Big| \,  |\beta|^d d\beta \leq 4 \int_{\R} {\rm e}^{- 4 a |\beta| } |\beta|^{d+1} d\beta < \infty\,,
\end{multline*}  
which by       Lebesgue's derivation theorem ensures that  the map $z  \longmapsto    H^1_{z}$ is holomorphic on the domain $\ds D\eqdefa \big\{z\in\C, {\rm Re}(z)  >0  \big\} $.

\medskip On the other hand,  we have by definition  for all $z$ in $\C^*$ 
$$  
H^2_{z} (Y,s)=  \frac {1}  {(4\pi z)^{\frac Q 2} }    \int_{\R}  \biggl(\frac {2 \tau}  {\sinh 2 \tau}\biggr)^{d} \, \exp\biggl( i \frac{\tau s}{2z}  -  \frac { |Y |^2 \tau}  {2 z \tanh  2 \tau} \biggr) d\tau  \, ,
$$
where  of course the integrated function is holomorphic on $\C^*$.  Now our aim is to apply Lebesgue's derivation theorem to establish that $H^2_{z} (Y,s)$ is holomorphic on some domain of $\C$.

\medskip Writing
 $$
  \frac { |Y |^2 \tau}  {2 z \tanh  2 \tau} =\frac {   |Y |^2 \tau }  {2  |z |^2 \tanh  2 \tau} \overline z  $$
 and  setting $z= |z | {\rm e}^{i  {\rm {arg}} (z)}$,
we readily gather that 
\beq\label{expestimate}
\biggl|\exp\biggl(-  \frac { |Y |^2 \tau }  {2 z \tanh  2 \tau} \biggr)\biggr| = \exp\biggl(\frac {-  |Y |^2 \tau }  {2  |z | \tanh  2 \tau}   \cos (  {\rm {arg}} (z))   \biggr)
\eeq
 and
\beq\label{dzexpestimate}
 \biggl|\partial_z   \exp\biggl( -  \frac { |Y |^2 |\tau| }  {2 z \tanh  2 \tau} \biggr) \biggr|= \frac { |Y |^2 \tau }  {2  |z |^2  \tanh  2 \tau}   \exp\biggl(\frac {-  |Y |^2 \tau }  {2  |z | \tanh  2 \tau}   \cos (  {\rm {arg}} (z))   \biggr)   \, \cdot
 \eeq
Along the same lines, one can easily check that 
$$ \biggl| \exp\biggl(\frac{i \tau s}{2z}\biggr)\biggr| =   \exp\biggl(\frac { \tau s}  {2|z | } \sin (  {\rm {arg}} (z))  \biggr) \, \virgp $$
 and
 $$\biggl|\partial_z  \exp\biggl(\frac{i \tau s}{2z}\biggr)\biggr| = \frac {|\tau| | s|}  {2 |z |^2 }  \, \exp\biggl(\frac { \tau s}  {2 |z | } \sin (  {\rm {arg}} (z))  \biggr) \cdot$$
We infer that    for all $\tau \in \R$,~$w=(Y,s) \in \H^d$ and all $z$ in $\C$ satisfying ${\rm Re}(z) \geq a >0 $, 
\beq\label{firstestimate}
\biggl|\exp\biggl( i \frac{\tau s}{2z}  -  \frac { |Y |^2 \tau }  {2 z \tanh  2 \tau} \biggr)\biggr| \leq  \exp\biggl(\frac {|\tau| | s|}  {2 |z | }  \biggr) 
\eeq
and
\beq\label{secondestimate}
 \biggl|\partial_z  \biggl(  \exp\biggl( i \frac{\tau s}{2z}  -  \frac { |Y |^2 \tau }  {2 z \tanh  2 \tau} \biggr) \biggr)\biggr| \leq   \exp\biggl(\frac {|\tau| | s|}  {2 |z | }  \biggr) \Big(\frac {1}  {a}  + \frac {|\tau| | s|}  { 2|z |^2 } \Big)  \, .
\eeq
  Fix $0< C < 4d$, then combining Formula\refeq{defH2} together with the Lebesgue derivation theorem, we deduce that the map~$z  \longmapsto    H^2_{z}$ is holomorphic on  \beq\label{defDs}
  \ds \wt D_{| s|}\eqdefa \Big\{z\in D,  |z |  > \frac  {| s|}  {C}  \Big\} \cdotp
\eeq

\medskip Since by  B. Gaveau's result (see Section~\ref{KH}), the maps  $H^1_z$ and $H^2_z$ coincide on the intersection of the real line with $\wt D_{| s|}$, we conclude that they also coincide on the whole domain $\wt D_{| s|}$.

\medskip Consider now  $(z_p)_{p \in \N}$   a sequence of elements of $\wt D_{| s|}$   which converges to~$-it$, with~$t \in \R^*$,  and let us   investigate~$\ds  \lim_{p \to \infty} \langle H^1_{z_p}, \!\varphi \rangle_{\cS'(\H^d)\times \cS(\H^d)}$ and~$\ds  \lim_{p \to \infty} \langle H^2_{z_p}, \!\varphi \rangle_{\cS'(\H^d)\times \cS(\H^d)}$,  for $\varphi$ in $\cS(\H^d)$.

\medskip
Let us start with $\ds \lim_{p \to \infty} \langle H^1_{z_p}, \!\varphi \rangle_{\cS'(\H^d)\times \cS(\H^d)}$.  By\refeq{defS'}, there holds   
 \beq
\label {computations1}
\langle H^1_{z_p}, \!\varphi \rangle_{\cS'(\H^d)\times \cS(\H^d)} =   \langle \cF_{\H} H^1_{z_p}, \!\theta \rangle_{\cS'(\wh \H^d)\times \cS(\wh \H^d)} \, ,
\eeq
with $\varphi (y,\eta,s)= c_d (\cF_{\H}^{-1}\theta)(y,-\eta,-s)$, for some constant $c_d.$ 
 
\medskip But according to\refeq{defH1}, 
$$
 \begin{aligned} \langle \cF_{\H} H^1_{z_p}, \theta \rangle_{\cS'(\wh \H^d)\times \cS(\wh \H^d)}&=  \int_{\wt \H^d} 
{\rm e}^{- 4 z_p |\lam| (2 |m|+d)} \delta_{n, m} \theta (\hat w) d\wh w \\ &= \sum_{m\in \N^{d} } \int_{\R} 
{\rm e}^{- 4 z_p |\lam| (2 |m|+d)} \theta (m, m, \lam) |\lam|^d d\lam 
\,  .
 \end{aligned} 
 $$
  Besides since~$\theta$ belongs to~$\cS(\wh \H^d)$,  it stems from\refeq{semipart}  that for any integer~$N$ there exists~$C_N$  such that for all~$\hat w=(m,m, \lam) \in \wt  \H^d$
$$|\theta (\hat w)| \leq C_N(1+  4|\lam|(2|m|+d))^{-N} \|\theta \|_{N,\cS(\wh\H^d)}\, .$$
Then performing  the change of variables  $\beta= \lam(2|m|+d)$ in each integral of the right-hand side of the above identity and choosing $N \geq d+2$, we get
\begin{multline*}  \sum_{m\in \N^{d} }  \int_{\R} \big |
{\rm e}^{- 4 z_p |\lam| (2 |m|+d)}  \,\theta(m, m, \lam) \big |\, |\lam|^d d\lam \\ \leq C_N\|\theta \|_{N,\cS(\wh\H^d)}\sum_{m\in \N^{d} }  \frac 1 {(2|m|+d)^{d+1}}\int_{\R} (1+  4|\beta|)^{-N} |\beta|^d d\beta < \infty \, .\end{multline*}
Applying  Lebesgue's dominated convergence theorem,   we infer  that\footnote{Let us underline that Formula\refeq{computations1} holds true for any sequence $(z_p)_{p \in \N}$     of $D$  which converges to~$-it$, with~$t \in \R^*$.}
 \beq
\label {computationslim11} \lim_{p \to \infty} \langle H^1_{z_p}, \!\varphi \rangle_{\cS'(\H^d)\times \cS(\H^d)} = \langle H^1_{-it}, \!\varphi \rangle_{\cS'(\H^d)\times \cS(\H^d)} = \langle S_t, \!\varphi \rangle_{\cS'(\H^d)\times \cS(\H^d)} \, .\eeq   In order to deal with $H^2_{z_p}$,     recall that by definition 
$$
  H^2_{z_p} (Y,s)=  \frac {1}  {(4\pi z_p)^{\frac Q 2} }    \int_{\R}  \biggl(\frac {2 \tau}  {\sinh 2 \tau}\biggr)^{d} \, \exp\biggl( i \frac{\tau s}{2z_p}  -  \frac { |Y |^2 \tau}  {2 z_p \tanh  2 \tau} \biggr) d\tau  \, .
$$
 By hypothesis, $(z_p)_{p \in \N}$ is a sequence of $D$ satisfying  $\ds  |z_p |>\frac  {|s |}  {C}  \virgp$ with $0< C < 4d$,  and this implies that  
 $$  \int_{\R}  \Big |\Big(\frac {2 \tau}  {\sinh 2 \tau}\biggr)^{d} \exp\biggl( i \frac{\tau s}{2z_p}  -  \frac { |Y |^2 \tau}  {2 z_p \tanh  2 \tau} \Big)  \Big|d\tau \leq \int_{\R} \Big(\frac {2 \tau}  {\sinh 2 \tau}\biggr)^{d} \exp\Big( \frac  {C |\tau| } 2\Big)d\tau < \infty \, .$$ 
We deduce  that for all $w=(Y,s)$ satisfying $|s |< 4d|t |$, there holds 
 $$\lim_{p \to \infty}   H^2_{z_p}(Y,s)= H^2_{-it}(Y,s) = \frac {1}  {(-4i \pi t)^{\frac Q2} } \int_{\R}  \Big(\frac {2 \tau}  {\sinh 2 \tau}\Big)^{d} \exp\Big( - \frac{\tau s}{2t}  - i \frac {|Y|^2\tau}  {2t \tanh  2 \tau} \Big) d\tau\, ,$$
which of course ensures that, for all $\varphi$ in $\cS(\H^d)$, we have $$\ds  \lim_{p \to \infty} \langle H^2_{z_p}, \!\varphi \rangle_{\cS'(\H^d)\times \cS(\H^d)}=\langle H^2_{-it}, \!\varphi \rangle_{\cS'(\H^d)\times \cS(\H^d)}\, .$$ This ends  the proof of Theorem\refer{STth}. \qed

%%%%%%%%%%%%

\section {Proof of the local dispersive and Strichartz estimates}
\label {proofmain1}
\subsection{Proof of Theorem~\ref{mainth1}}
Since the linear Schr\"odinger  equation  on $\H ^d$ is invariant by left translation, it suffices to prove the result for~$w_0 = 0$.
Let $u_0$ be a function  in  $\cD(B_\H(0,  R_0))$.
Then  invoking Theorem\refer{STth}, we infer that  the solution to  the Cauchy problem~$(S_\H)$   assumes the form 
\begin{equation}\label{solution Schrodinger}
 u(t, \cdot)= u_0 \star S_t \, ,
 \end{equation}
where  
\begin{equation}\label{kernelSchrodinger}
S_t(Y,s)=\frac {1}  {(-4i \pi t)^{\frac Q 2} } \int_{\R}  \biggl(\frac {2 \tau}  {\sinh 2 \tau}\biggr)^{d} \, \exp\biggl( - \frac{\tau s}{2t}  - i \frac {|Y|^2\tau}  {2t \tanh  2 \tau} \biggr) d\tau\,  
 \end{equation} 
  on any Heisenberg  ball $B_\H(0, \kappa\sqrt{|t |})$, with $\kappa< \sqrt{4d}$, since
  $$
  \rho_\H(w) \leq  \kappa\sqrt{|t|} \Longrightarrow |s|\leq  \kappa^2 |t| < 4d |t| \, .
    $$
  Note that
  \begin{equation} \label{Linfty} 
\|S_t\|_{L^\infty(B_\H(0, \kappa\sqrt{|t |}))} \leq  \frac {1}  {(4 \pi |t |)^{\frac Q 2} } \int_{\R} \Big(\frac {2 \tau}  {\sinh 2 \tau}\Big)^{d} \exp\Big(\frac { \kappa^2|\tau|} 2\Big)d\tau  \eqdefa   \frac {M_\kappa}  {|t |^{\frac Q 2} }\,  \cdot 
\end{equation}
\medskip But by definition of the convolution product on $\H^d$, we have
\begin{equation}\label{convolution Schrodinger}
(u_0 \star S_t) (w)  
= \int_{\H^d} u_0 ( v ) S_t ( v^{-1} \cdot w)\, dv\, .
 \end{equation} 
 Thanks to  the triangle inequality, we have   for any  $w$ in $B_\H(0, \kappa\sqrt{|t |})$  and any $v$ in $B_\H(0,  R_0)$  
 $$
 \rho(v^{-1} \cdot w) = d_\H (w,v) \leq \rho(w) + \rho(v) \leq \kappa\sqrt{|t |} + R_0  <    \sqrt{4d|t |} \, ,
 $$
provided that $$ |t | > T_{\kappa,R_0}= \Big(\frac {R_0} {\sqrt{4d}-\kappa} \Big)^2 \virgp
$$
  which  by    Young's inequality completes the proof of Estimate\refeq{dispSHloc}   in the case~$p=\infty$.  

\medskip

Furthermore by the  conservation of the mass there holds
$$
\|u(t,\cdot)\|_{L^2(B_\H(0, \kappa\sqrt{|t |}))} \leq \|u(t, \cdot)\|_{L^2(\H^{d})} = \|u_0\|_{L^2(\H^{d})}.
$$
So resorting to a real  interpolation argument, we get  for all~$2 \leq p \leq \infty$ and any~$|t | \geq T_{R_0,\kappa}$
$$
\|u(t,\cdot)\|_{L^p(B_\H(0, \kappa\sqrt{|t |}))} \leq \left( \frac {M_\kappa}  {|t |^{\frac Q 2} }\right)^ {1- \frac 2 p}   \|u_0\|_{L^{p'}(\H^{d})}\, ,
$$
 where $p'$ denotes the conjugate exponent of $p$. This ends the proof of the result.  
\qed

\subsection{Proof of Theorem\refer{STinvth}}
As already mentioned,  the Strichartz estimates are straightforward from the dispersive estimates.  Actually, Theorem\refer{STinvth} readily follows from  the  mass conservation and   the following proposition, which can be seen as a corollary of Theorem~{\rm\ref{mainth1}}. 
   \begin{proposition}
\label{mainth1st+}
{\sl  Under the assumptions of Theorem~{\rm\ref{mainth1}},  the solution to the Cauchy problem~$(S_\H)$ associated to $u_0$ satisfies, for all $2 \leq p \leq \infty$ and all $q$ such that $\ds \frac 1  q + \frac Q  p < \frac Q  2\virgp$ 
 
$$ \|u\|_{L^q(]-\infty, - C_\kappa R_0^2]  \cup [C_\kappa R_0^2, \infty[; L^p(B_\H(w_0, \kappa\sqrt{|t |})))}  \leq
 C(q,\kappa) {R_0^{-Q(\frac  1 2 - \frac 1 p)+ \frac 2 q}}    \|u_0\|_{L^{2}(\H^{d})}\, .
  $$
}
\end{proposition}
\begin{proof}
Since $u_0$ is  in supported in $B_\H(w_0, R_0)$, combining  the H\"older inequality   with\refeq{dispSHloc}, we infer that, for all $2 \leq p \leq \infty$,  \beq
\label {dispSHlocl2}
\|u(t,\cdot)\|_{L^p(B_\H(w_0, \kappa\sqrt{|t |}))} \leq C(\kappa)
\frac {R_0^{Q(\frac  1 2 - \frac 1 p)}}  {|t |^{  \frac Q 2 - \frac Q p } }  \|u_0\|_{L^{2}(\H^{d})}\, \virgp
\eeq for all $|t | \geq T_{\kappa, R_0}$. 
The proposition follows after time integration, and Theorem \refer{STinvth} is a direct consequence.
\end{proof}
   
\section {Concentration properties of the Schr\"odinger kernel  on~$\H^d$}\label {kerneldirac}
\subsection{Proof of Theorem\refer{STthdirac}} The proof  relies on Fourier-Heisenberg analysis recalled in Section\refer{fouriertransf}.   
 \medskip 
 
Denote  $w_t=(0,  -4dt)$. Then,    making  use of  Formula\refeq{solSHkernel} and recalling that $S_t$ is even,  we get, for any  $u_0 \in \cS(\H^d)$,  
\begin{equation}\label{kernelSchrodingerconvsol} (u_0 \star S_t)(w_t)=\langle S_t,   u_0  \circ\tau_{w^{-1}_t}\rangle_{\cS'(\H^d)\times \cS(\H^d)} = \sum_{n\in \N^d}\int_{\R}{\rm e}^{4 i t |\lam|(2|n|+d)}   \theta_t (n,n,\lam)|\lam|^dd\lam\,  ,\end{equation}where $\tau_{w^{-1}_t}$ is the  left translation operator by $w^{-1}_t$ defined by\refeq{definlefttranslate} and   \beq\label{reltheta}
   {}^t\cF_{\H}\theta_t(y,\eta,s) = \frac {\pi^{d+1}} {2^{d-1}} (\cF_{\H}^{-1}\theta_t)(y,-\eta,-s)=  u_0 \circ\tau_{w^{-1}_t}\, . 
\eeq 
Notice that, for all $f  \in L^1(\H^d)$ and all $g_0=(0, s_0)$, we have
$$
\cF_\H (f\circ \tau_{g^{-1}_0}) (n, m, \lam) =   {\rm e}^{-is_0\lam} (\cF_\H f)(n, m, \lam) \, .
$$
It then follows  that
$$
\cF_{\H}(u_0 \circ\tau_{w^{-1}_t})(n, m, \lam)  = {\rm e}^{4idt\lam} \cF_\H (u_0) (n, m, \lam)    \, .
$$
Combining\refeq{defF}-\eqref{def} together with\refeq{reltheta}, we deduce that
$$ \theta_t (n, m, \lam) = \frac  {2^{d-1}} {\pi^{d+1}}  {\rm e}^{- 4idt\lam}\cF_\H  (u_{0} )(m, n, -\lam)  \, .$$
Consequently, we have 
$$ \begin{aligned} (u_0 \star S_t)(w_t) = \frac  {2^{d-1}} {\pi^{d+1}} \sum_{n\in \N^d}\int_{\R}{\rm e}^{4idt\lam} {\rm e}^{4 i t |\lam|(2|n|+d)}    \cF_\H   (u_{0} )(n, n, \lam)|\lam|^dd\lam \, .\end{aligned}$$
In particular, if we consider $u_0$ so that 
$$\cF_\H   (u_0)(n, n, \lam)=\cF_\H   (u_0)(n, n, \lam) \delta_{n, 0}{\bf 1}_{\lam <0}  \, ,
$$
we obtain, thanks to\refeq{DiracF},
$$ \begin{aligned} (u_0 \star S_t)(w_t) & = \langle \d_0, u_0\rangle_{\cS'(\H^d)\times \cS(\H^d)} \, .\end{aligned}$$ 
Theorem\refer{STthdirac} follows. \qed

\subsection{Proof of Theorem\refer{better}}
Let~$\ell \geq 1$ be a fixed integer.  We revisit the proof of Theorem~\ref{STth}: recall that the restriction over~$s$
comes from the fact that the function~$H^2_z(Y,s)$ given by\refeq{defH2} is holomorphic only on the set~$\ds  |z |  >    {| s|} / ({4d})$. 
Now let us define
$$
\mathscr{V}_\ell (\H^d)\eqdefa\oplus_{|m|\geq \ell}  L_m^2(\H^d)
$$
where~$L^2_m$ is defined in~(\ref{defL2m})-(\ref{l2m}). In the case when the Cauchy data~$u_0$ belongs to the set~$\cS(\H^d) \cap \mathscr{V}_\ell (\H^d)$,
our goal is to write
$$
u(t)= u_0 \star   S^{(\ell)}_t
$$
where~$S^{(\ell)}_t$ is  a tempered distribution obtained, for~$|s|<4(2\ell+ d)|t|$, by the same complex analysis argument 
as in the proof of Theorem~\ref{STth}, where the    function to be analyzed,  arising from the heat equation, is now
 $$h^{(\ell)}_t(Y,s)\eqdefa
  \frac {2^{d-1}}  {\pi^{d+1} }  \sum_{|m|\geq \ell} \int_{\R\times \R^{d}} 
{\rm e}^{is\lam+2i\lam \langle \eta,z\rangle} {\rm e}^{- 4 |\lam| (2 |m|+d)t}
H_{m,\lam} (y+z)
H_{m,\lam} (-y+z) \,dz |\lam|^d d\lam\,.$$
According to Gaveau's resulted recalled in Theorem\refer{Heatkernel}, the function $h^{(\ell)}_t$ also reads $$
  \begin{aligned}  
&  h^{(\ell)}_t(Y,s)= \frac {1}  {(4\pi t)^{\frac Q 2} } \int_{\R}  \biggl(\frac {2 \tau}  {\sinh 2 \tau}\biggr)^{d} \, {\rm e}^{ i \frac{\tau s}{2t}   -  \frac {|Y |^2\tau}  {2 t \tanh  2 \tau} } d\tau
\\
& \quad -  \frac {2^{d-1}}  {\pi^{d+1} }  \sum_{|m|\leq \ell-1} \int_{\R\times \R^{d}} 
{\rm e}^{is\lam+2i\lam \langle \eta,z\rangle} {\rm e}^{- 4 |\lam| (2 |m|+d)t}
H_{m,\lam} (y+z)
H_{m,\lam} (-y+z) \,dz |\lam|^d d\lam\,.
\end{aligned} 
$$
But in view of\refeq{princident},  we have, for any integer $k$,
$$
  \begin{aligned}  
&    \frac {2^{d-1}}  {\pi^{d+1} }  \sum_{|m|=k} \int_{\R\times \R^{d}} 
{\rm e}^{is\lam+2i\lam \langle \eta,z\rangle} {\rm e}^{- 4 |\lam| (2 |m|+d)t}
H_{m,\lam} (y+z)
H_{m,\lam} (-y+z) \,dz |\lam|^d d\lam\\
&=   \frac {2^{d-1}}  {\pi^{d+1} } \int_{\R} 
{\rm e}^{is\lam} {\rm e}^{- 4 |\lam| (2 k+d)t}
{\rm e}^{- |\lam||Y|^2 } L_k^{(d-1)} ( 2  |\lam|  |Y|^2)|\lam|^d d\lam\,,
\end{aligned} 
$$
where  $L_k^{(d-1)}$ denotes   the  Laguerre polynomial of order $k$ and type $d-1$. Then,   performing    the change   of variables $\ds \lam= \frac \tau {2t}\virgp$  we readily gather that 
$$
  \begin{aligned}  
  h^{(\ell)}_t(Y,s)&= \frac {1}  {(4\pi t)^{\frac Q 2} } \int_{\R}  \biggl(\frac {2 \tau}  {\sinh 2 \tau}\biggr)^{d} \, {\rm e}^{ i \frac{\tau s}{2t}   -  \frac {|Y |^2\tau}  {2 t \tanh  2 \tau}} d\tau
\\
& \quad - \frac {1}  {(4\pi t)^{\frac Q 2} }  \sum_{k\leq \ell-1} \int_{\R} \big(4 |\tau|\big)^d {\rm e}^{- 2 |\tau| (2k+d) }
L_k^{(d-1)} \Big(\frac {|Y |^2|\tau|}  { t } \Big) {\rm e}^{  i \frac{\tau s}{2t}   -  \frac {|Y |^2|\tau|}  {2 t } }  d\tau\,.
\end{aligned} 
$$
Returning again to the strategy of the proof of Theorem~\ref{STth},
our first aim is therefore to prove that the maps
$$z    \longmapsto   H^{(\ell), 1}_z(Y,s)  \andf z    \longmapsto   H^{(\ell), 2}_{z}(Y,s) $$
with 
$$\ds  H^{(\ell), 1}_z(Y,s) \eqdefa \ds \frac {2^{d-1}}  {\pi^{d+1} }  \sum_{|m|\geq \ell} \int_{\R} 
{\rm e}^{i s \lam} {\rm e}^{- 4 z |\lam| (2 |m|+d)} \cW\big((m,m,\lam),Y\big) |\lam|^d d\lam 
$$
and
\beq\label{defH2l}
 \begin{aligned}  
H^{(\ell), 2}_z(Y,s)&\eqdefa \frac {1}  {(4\pi z)^{\frac Q 2} } \int_{\R} (2 |\tau|)^d {\rm e}^{ i \frac{\tau s}{2z}   }\biggl(  \frac {1}  {(\sinh 2 |\tau|)^d }  \, {\rm e}^{   -  \frac {|Y |^2 \tau}  {2z \tanh  2  \tau}}  \\
& \qquad\qquad\qquad -  \sum_{k\leq \ell-1}  2^d {\rm e}^{- 2 |\tau| (2k+d) }
L_k^{(d-1)} \Big(\frac {|Y |^2|\tau|}  { z } \Big) {\rm e}^{   -  \frac {|Y |^2|\tau|}  {2 z } }  \biggr)d\tau 
\end{aligned} 
\eeq
are,  for   all $(Y,s)$ in $\H^d$, holomorphic on a suitable domain of $\C^*$. The same reasoning  as in the proof of  Theorem~\ref{STth}  enables to check  that the function $H^{(\ell), 1}_z$  is holomorphic on the domain~$\ds D=\big\{z\in\C, {\rm Re}(z)  >0  \big\}$ so now we  concentrate on~$H^{(\ell), 2}_z$.
\medskip

We shall prove  that the function~$H^{(\ell), 2}_{z}$    is holomorphic on the domain
 \beq\label{defDsl}
  \ds \wt D^ \ell_{| s|}\eqdefa  \Big\{z\in D,  |z |  > \frac  {| s|}  {4(2\ell+d)}  \Big\} \, \cdotp
\eeq
Let us start by re-writing~(\ref{defH2l}) in the following form:
\beq\label{defH2bis}
  \begin{aligned}  
H^{(\ell), 2}_z(Y,s)&= \frac {1}  {(4\pi z)^{\frac Q 2} } \int_{\R}  (4|\tau|)^d{\rm e}^{ i \frac{\tau s}{2z} -2|\tau|d}\, 
     \Phi_\ell\Big( \frac {|Y |^2|\tau|}  { z }\virgp \, {\rm e}^{-4|\tau|}\Big)  \, d\tau \, , \quad \mbox{with}\\
  \Phi_\ell(x,r)& \eqdefa  {\rm e}^{ -\frac x2} \Big(\frac{{\rm e}^{- \frac {  r  x }  { 1-  r}} }{(1-  r)^d}- \sum_{k\leq \ell-1}r^{k }L_k^{(d-1)}(x)\Big)
   \, .
\end{aligned} 
\eeq
From now on we set\footnote{It will be useful to point out that $\ds  |x| = \frac {|Y |^2|\tau|}  { |z| } $ and $\ds {\rm Re}(x) = \frac {|Y |^2|\tau|}  { |z| } \cos (  {\rm {arg}} (z))$. }
\beq\label{defrx}
 x \eqdefa \frac {|Y |^2|\tau|}  { z } \,,     \quad r \eqdefa {\rm e}^{-4|\tau|} \, . 
\eeq
Recalling that   the generating function for the  Laguerre polynomials  is  given by  (see for instance\ccite{askey, Tricomi, farautharzallah, gradshteyn, lebedev, O}) 
\beq\label{genf}
\sum_{k\geq 0} 
 r^{k} L_k^{(d-1)}(x)=\frac{{\rm e}^{- \frac {  r  x }  { 1-  r}} }{(1-  r)^d}\, \virgp \quad   |r|< 1
\, ,
\eeq 
we notice
 that the~${\rm e}^{ \frac x2} \Phi_\ell(x,r)$  is nothing else than the remainder of the Taylor expansion of the function~$\ds f(x, r)\eqdefa \frac{{\rm e}^{- \frac {  r  x }  { 1-  r}} }{(1-  r)^d}$ at order~$\ell-1$, near~$r=0$. We shall therefore  argue differently depending on whether~$r$ is close to~$1$ or not.   So let us fix~$\tau_0>0$, and   start by analyzing the case when~$|\tau| \geq \tau_0 $, since this implies that~$r\leq r_0\eqdefa {\rm e}^{-4\tau_0}<1$. The case~$|\tau| \leq  \tau_0$ will be dealt with further down.
Considering
\beq\label{defHtau0}
H^{(\ell), 2}_{z,\tau_0}(Y,s)\eqdefa \frac {1}  {(4\pi z)^{\frac Q 2} } \int_{|\tau| \geq \tau_0 } 
(4 |\tau|)^d {\rm e}^{ i \frac{\tau s}{2z}  -2|\tau|d } \,    \Phi_\ell \Big(\frac {|Y |^2|\tau|}  {  z }\virgp  \,{\rm e}^{- 4 |\tau|}\Big) d\tau  \, ,
\eeq
we are thus reduced to investigating, for~$|\tau| \geq \tau_0$,  the function
\beq\label{newfGint}
\big(4 |\tau|\big)^d {\rm e}^{- 2 |\tau| d } {\rm e}^{i \frac{\tau s}{2z}} \,   \Phi_\ell \Big(\frac {|Y |^2|\tau|}  {  z }\virgp  \,{\rm e}^{- 4 |\tau|}\Big) 
\eeq   
 and its derivative with respect to $z$. We shall actually restrict~$z$ to    the domain
 $$
 \begin{aligned}  
\ds D^\ell _{| s|, a,A } &\eqdefa \Big\{z\in \C\, ,\,    |z |  > \frac  {| s|}  {\kappa (2\ell+d)}  \Big\} \cap   \Omega_{a, A}\, ,  \,  \mbox{with}  \\
   \Omega_{a, A} &\eqdefa \big\{z\in\C, {\rm Re}(z)  >a \, , \,  |z| \leq A  \big\} \, , 
 \end{aligned} 
$$
where~$0<\kappa<4$ and  $a,A>0$ are arbitrary fixed constants.

Recalling~$\ds f(x, r)\eqdefa \frac{{\rm e}^{- \frac {  r  x }  { 1-  r}} }{(1-  r)^d}$ and applying Taylor's formula, it readily follows from\refeq{genf} that 
$$
\Phi_\ell(x,r)=  {\rm e}^{-\frac x 2}  \frac{r^\ell  }{(\ell-1)!}  \int_0^1  (1-s)^{\ell-1}  (\partial_r^\ell f) (x, rs) \, ds \, .
$$
 After some computations we infer that  there is a positive constant $C(\ell, \tau_0)$ such that, for all~$r\leq r_0={\rm e}^{-4\tau_0}$, we have 
\beq\label{newfGintbis}  \begin{aligned}  
 \big|\Phi_\ell(x,r)\big|  &\leq C(\ell, \tau_0)   \, r^\ell  \,  {\rm e}^{-\frac {{\rm Re}(x)} 2}  (1+|x|)^\ell  \andf  \\
   \big|\partial_x \Phi_\ell(x,r)\big|   &\leq C(\ell, \tau_0)   \, r^\ell  \,  {\rm e}^{-\frac {{\rm Re}(x)} 2}  (1+|x|)^{\ell } \, . \end{aligned}  
\eeq
Observing that 
\beq\label{estimatesmodulus}
  \big|{\rm e}^{   -  \frac {|Y |^2|\tau|}  {2 z } } \big| = {\rm e}^{   -  \frac {|Y |^2|\tau|}  {2 |z| } \cos (  {\rm {arg}} (z))}\, ,
  \eeq  with $\ds  \inf_{z \in \Omega_{a,A}} \cos ( {\rm {arg}} (z))  \geq \alpha(a,A) >0$, we deduce that there is a positive constant~$C(\ell, \tau_0, a, A)$ such that, for all $(Y,s)$ in~$\H^d$,~$\ds z \in D^\ell _{| s|, a,A }$ and~$|\tau| \geq  \tau_0$,  there holds 
  $$ 
  \begin{aligned}  
  \Big|(4|\tau|)^d{\rm e}^{ i \frac{\tau s}{2z} -2|\tau|d}\Phi_\ell  \Big(\frac {|Y |^2|\tau|}  { z }
\virgp \, {\rm e}^{- 4 |\tau|}\Big) \Big
  |  
       &   \leq C(\ell, \tau_0, a, A) |\tau|^ {d}    \, {\rm e}^{ \frac {-(4 -\kappa) (2\ell+d) |\tau|  }  {2  }  } 
\end{aligned}   $$
and similarly
$$
 \Big|\partial_z  \Big(4|\tau|)^d{\rm e}^{ i \frac{\tau s}{2z} -2|\tau|d}\Phi_\ell  \Big(\frac {|Y |^2|\tau|}  { z }
\virgp \, {\rm e}^{- 4 |\tau|}\Big)  \Big| \leq   C(\ell, \tau_0, a, A) |\tau|^ {d}    \, {\rm e}^{ \frac {-(4 -\kappa) (2\ell+d) |\tau|  }  {2  }  }.
$$
This readily ensures   that the function~$H^{(\ell), 2}_{z,\tau_0}$ 
is holomorphic on the domain~$\ds \wt D^ \ell_{| s|}$ defined in~(\ref{defDsl}).

\medskip

To deal with~$H^{(\ell), 2}_{z}-H^{(\ell), 2}_{z,\tau_0}$, let us first observe  that according to\refeq{defH2l}, we have 
$$ \begin{aligned}  
H^{(\ell), 2}_{z,\tau_0}(Y,s)&= \frac {1}  {(4\pi z)^{\frac Q 2} } \int_{|\tau| \leq \tau_0 } (2 |\tau|)^d {\rm e}^{ i \frac{\tau s}{2z}   }\biggl(  \frac {1}  {(\sinh 2 |\tau|)^d }  \, {\rm e}^{   -  \frac {|Y |^2 \tau}  {2z \tanh  2  \tau}}  \\
& \qquad\qquad\qquad -  \sum_{k\leq \ell-1}  2^d {\rm e}^{- 2 |\tau| (2k+d) }
L_k^{(d-1)} \Big(\frac {|Y |^2|\tau|}  { z } \Big) {\rm e}^{   -  \frac {|Y |^2|\tau|}  {2 z } }  \biggr)d\tau \, .
\end{aligned} 
$$
Then invoking\refeq{estimatesmodulus}, 
 we infer  that for any integer~$\ell$
and  all positive real numbers~$a$ and~$A$, there exists a positive constant~$C(\ell, a, A)$ such that  
$$ \begin{aligned}  
 \supetage {z \in \Omega_{a,A}} {k \leq \ell-1}  \Big| L_k^{(d-1)} \Big(\frac {|Y |^2|\tau|}  { z } \Big){\rm e}^{-\frac {|Y |^2|\tau|}  { 2z }}\Big|  &\leq C(\ell, a, A)    \andf 
  \\
\supetage {z \in \Omega_{a,A}} {k \leq \ell-1}  \Big|\frac d {dz } \Big(L_k^{(d-1)} \Big(\frac {|Y |^2|\tau|}  { z } \Big){\rm e}^{-\frac {|Y |^2|\tau|}  { 2z }}\Big)\Big|   &\leq C(\ell, a,A)     \, .
\end{aligned} $$
We deduce that  for all $(Y,s)$ in~$\H^d$,~$z\in D^\ell _{| s|, a,A }$   and $|\tau| \leq \tau_0$,    there holds \begin{equation}\label{sumon01} \biggl|\sum_{k\leq \ell-1}  \big(4 |\tau|\big)^d {\rm e}^{- 2 |\tau| (2k+d) }
L_k^{(d-1)} \Big(\frac {|Y |^2|\tau|}  { z } \Big) {\rm e}^{  i \frac{\tau s}{2z}   -  \frac {|Y |^2|\tau|}  {2 z } }  \biggr|  \leq C(\ell, a,A) \,\tau^d_0 \, {\rm e}^{ \frac {\tau_0 \kappa (2\ell+d)}  {2 } }\end{equation}
and 
\begin{equation}\label{sumon02} 
\begin{aligned} &   \biggl|\partial_z  \biggl( \sum_{k\leq \ell-1}  \big(4 |\tau|\big)^d {\rm e}^{- 2 |\tau| (2k+d) }
L_k^{(d-1)} \Big(\frac {|Y |^2|\tau|}  { z } \Big) {\rm e}^{  i \frac{\tau s}{2z}   -  \frac {|Y |^2|\tau|}  {2 z } }  \biggr)\biggr| \\ &\qquad\qquad\qquad\qquad    \leq C(\ell, a,A) {\rm e}^{ \frac {\tau_0 \kappa (2\ell+d)}  {2 } }\Big( 1+ \frac {\tau_0\kappa (2\ell+d)}  { 2  } \Big)\,  \cdotp
 \end{aligned} \end{equation}
This obviously implies that the function $$\frac {1}  {(4\pi z)^{\frac Q 2} } \sum_{k\leq \ell-1}  \int_{|\tau| \leq \tau_0 } (4 |\tau|)^d {\rm e}^{ i \frac{\tau s}{2z}   }    {\rm e}^{- 2 |\tau| (2k+d) }
L_k^{(d-1)} \Big(\frac {|Y |^2|\tau|}  { z } \Big) {\rm e}^{   -  \frac {|Y |^2|\tau|}  {2 z } }d\tau $$
is holomorphic on the domain~$\ds D^\ell _{| s|, a,A }$.  The part
$$\frac {1}  {(4\pi z)^{\frac Q 2} } \int_{|\tau| \leq \tau_0 } \biggl(\frac {2 \tau}  {\sinh 2 \tau}\biggr)^{d} \, {\rm e}^{ i \frac{\tau s}{2z}   }   \, {\rm e}^{   -  \frac {|Y |^2 \tau}  {2z \tanh  2  \tau}} $$ can   easily be dealt with, which  achieves the proof of the fact that the function~$H^{(\ell), 2}_{z}$    is holomorphic on~$\ds \wt D^ \ell_{| s|}$.

\medskip
Finally,  let~$(z_p)_{p \in \N}$ be a sequence  in  $\widetilde D^ \ell _{| s|}$   which converges to~$-it$, with~$t \in \R^*$.  Then arguing as in the proof of Theorem~\ref{STth}, one can readily gather that for  any function~$\varphi$ in~$\cS(\H^d) \cap \mathscr{V}_\ell(\H^d)$, we have  
$$
  \begin{aligned} 
  \lim_{p \to \infty} \langle H^{(\ell), 1}_{z_p}, \!\varphi \rangle_{\cS'(\H^d)\times \cS(\H^d)} &= \langle H^{(\ell), 1}_{-it}, \!\varphi \rangle_{\cS'(\H^d)\times \cS(\H^d)} \\ 
  &  = \langle S^{(\ell )}_t, \!\varphi \rangle_{\cS'(\H^d)\times \cS(\H^d)}= \langle S_t, \!\varphi \rangle_{\cS'(\H^d)\times \cS(\H^d)}\, .
\end{aligned}
$$
To show that   $$
  \lim_{p \to \infty} \langle H^{(\ell), 2}_{z_p}, \!\varphi \rangle_{\cS'(\H^d)\times \cS(\H^d)} = \langle H^{(\ell), 2}_{-it}, \!\varphi \rangle_{\cS'(\H^d)\times \cS(\H^d)} \,,
$$  we shall as above 
investigate separately~$H^{(\ell), 2}_{z,\tau_0}$ and~$H^{(\ell), 2}_{z}-H^{(\ell), 2}_{z,\tau_0}$. 
Let~$(z_p)_{p \in \N}$ be  a sequence in $\widetilde D^ \ell _{| s|}$   which converges to~$-it$, with~$t \in \R^*$, and let us start by  studying the part corresponding to~$H^{(\ell), 2}_{z,\tau_0}$. One can assume without loss of generality that  $\ds  \frac  {| s|}  {|z_p | } \leq \kappa (2\ell+d)$, with $0<\kappa<4$, and also  that~$\ds (1- \delta)  |t| \leq |z_p| \leq (1+ \delta) |t| $  for some small~$\delta$.  Then  taking advantage of Estimate\refeq{newfGintbis}, we readily gather  that
$$ \begin{aligned} &  \big|\varphi (Y,s) \big| \,  \Big|\big(4 |\tau|\big)^d {\rm e}^{- 2 |\tau| d } {\rm e}^{i \frac{\tau s}{2z_p}} \,   \Phi_\ell \Big(\frac {|Y |^2|\tau|}  {  z_p }\virgp  \,{\rm e}^{- 4 |\tau|}\Big) \Big| \\ &\qquad \qquad \qquad \qquad  \leq C(\ell,\delta, \tau_0)  \big|\varphi (Y,s) \big|    \Big(1+ \frac {|Y |^2}  { (1-\delta)  |t| } \Big)^\ell\ |\tau|^{d+ \ell}{\rm e}^{\frac {(\kappa-4) (2\ell+d) |\tau|  }  {2  } }\, \virgp\end{aligned}$$ which implies that   $$
  \lim_{p \to \infty} \langle H^{(\ell), 2}_{z_p,\tau_0}, \!\varphi \rangle_{\cS'(\H^d)\times \cS(\H^d)} = \langle H^{(\ell), 2}_{-it,\tau_0}, \!\varphi \rangle_{\cS'(\H^d)\times \cS(\H^d)} \,.
$$  Now to study  the part corresponding to~$H^{(\ell), 2}_{z}-H^{(\ell), 2}_{z,\tau_0}$,  let us first observe  that it stems from\refeq{estimatesmodulus} that $|\exp(-  \frac {|Y |^2|\tau|}  {2 z_p}| \leq 1$. Consequently,   there exists  a positive constant~$C(\ell,\delta, \tau_0)$ such that,  for all $(Y,s)$ in~$\H^d$  and $|\tau| \leq \tau_0$,    there holds
  $$\begin{aligned} &  \big|\varphi (Y,s) \big| \, \Big|\sum_{k\leq \ell-1}  \big(4 |\tau|\big)^d {\rm e}^{- 2 |\tau| (2k+d) }
L_k^{(d-1)} \Big(\frac {|Y |^2|\tau|}  { z_p } \Big) {\rm e}^{  i \frac{\tau s}{2z_p}   -  \frac {|Y |^2|\tau|}  {2 z_p } } \Big|  \\ & \qquad \qquad \qquad \qquad \leq C(\ell,\delta, \tau_0) \big|\varphi (Y,s) \big| \Big(1+ \frac {|Y |^2|\tau_0|}  { (1-\delta)  |t| } \Big)^\ell\, 
 {\rm e}^{\frac {\kappa (2\ell+d) \tau_0  }  {2  } } \, \cdotp\end{aligned}$$
 Since the first part can be easily dealt, we readily gather that   $$
  \lim_{p \to \infty} \langle H^{(\ell), 2}_{z_p}-H^{(\ell), 2}_{z_p,\tau_0}, \!\varphi \rangle_{\cS'(\H^d)\times \cS(\H^d)} = \langle H^{(\ell), 2}_{-it}-H^{(\ell), 2}_{-it,\tau_0}, \!\varphi \rangle_{\cS'(\H^d)\times \cS(\H^d)} \,.
$$
 This ends   the proof of  Theorem\refer{better}.
 \qed

\end{document}